\documentclass[10pt,a4paper,twoside]{article}

\usepackage[latin1]{inputenc}
\usepackage{amssymb}
\usepackage{amsmath}
\usepackage{amsfonts}
\usepackage{amsbsy}
\usepackage[square,sort,comma,numbers]{natbib}
\usepackage{eucal}
\usepackage{graphicx}
\usepackage{epstopdf}
\usepackage[english]{babel}
\usepackage{mathptmx}
\usepackage{float}
\usepackage[font=small]{caption}
\usepackage{array}
\usepackage{amsthm}
\usepackage{comment}
\usepackage{setspace}
\usepackage{bm}
\newtheorem{theorem}{Theorem}[section]

\newtheorem{lemma}[theorem]{Lemma}

\newtheorem{charact}[theorem]{Characterization}
\begin{document}
\title{New consistent exponentiality tests based on V-empirical Laplace transforms with comparison of efficiencies}

\author{Marija Cupari\'c\thanks{marijar@matf.bg.ac.rs},  Bojana Milo\v sevi\'c\footnote{bojana@matf.bg.ac.rs (corresponding author)} , Marko Obradovi\' c \footnote{marcone@matf.bg.ac.rs}
         \\\medskip
{\small Faculty of Mathematics, University of Belgrade, Studenski trg 16, Belgrade, Serbia}}

\date{}

\maketitle
\begin{abstract}
We present new  consistent goodness-of-fit tests for exponential
distribution,  based on the Desu characterization. The test statistics
represent the weighted $L^2$ and $L^{\infty}$ distances between
appropriate V-empirical Laplace transforms of random variables that appear
in the characterization. In addition, we perform an extensive comparison
of Bahadur efficiencies of different recent and classical exponentiality
tests.
We also present the empirical powers of new tests.
\end{abstract}
{\small \textbf{ keywords:}  goodness-of-fit; exponential distribution; Laplace transform; Bahadur efficiency; V-statistics 

}

\section{Introduction}
To justify the use of more complicated models for lifetime data, one of the first steps is to reject the most simple one, the exponential.
For this purpose numerous tests have been developed and are available in the literature.

The classical, and most commonly used procedure, is to apply one of  universal time-honored goodness-of-fit tests based on empirical distribution function, such as Kolmogorov-Smirnov, Cramer-von Mises, Anderson-Darling. To make them applicable to the case of a composite null hypothesis, the  Lilliefors modification with estimated rate parameter is frequently used.

Another approach is to use tests tailor-made for testing exponentiality. Such tests usually employ some special properties of the exponential distribution. Various integral transform related properties have been exploited: 
 characteristic functions (see e.g. \cite{henze1992new}, \cite{henze2002goodness}, \cite{henze2005}); Laplace transforms (see e.g. \cite{henze2002tests}, \cite{klar2003test}, \cite{meintanis2007testing}); and other integral transforms (see e.g. \cite {klar2005tests}, \cite{meintanis2008tests}). Other possible properties include maximal correlations (see \cite{grane2009location}, \cite{grane2011directional}, \cite{strzalkowska2017goodness}), entropy  (see \cite{alizadeh2011testing}), etc.

An important type of such properties are the characterizations of the exponential distribution. Many of them, being relatively simple, are very suitable for construction of goodness-of-fit tests. This is especially true for the equidistribution-type characterizations. Since the equality in distribution can be expressed in many ways (equality of  distribution functions, densities, integral transforms, etc.),  many different types of test statistics are available.
Tests that use U-empirical and V-empirical distribution functions, of  integral-type (integrated difference) and supremum-type, can be found in \cite{NikVol}, \cite{volkova2015goodness}, \cite{jovanovic}, \cite{Publ}, \cite{bojanaMetrika}, \cite{nikitin2016efficiency}. Weighted integral-type and $L^2$-type tests that use U- or V- empirical Laplace transforms are presented in \cite{MilosevicObradovicPapers} and in \cite{cuparic2018new}.

The common approach to explore the quality of tests is to find their power against different alternatives. Several  papers are devoted to comparative power studies of exponentiality tests (see e.g. \cite{henze2005}, \cite{torabi2018wide}, \cite{allison2017apples}).

Another popular choice for the quality assessment is the asymptotic efficiency. In this regard, however, no extensive study has been done.
In this paper our aim is to compare the exponentiality tests using the approximate Bahadur efficiency (see \cite{bahadur1960asymptotic}).

We opt for the approximate Bahadur efficiency since it is applicable to asymptotically non-normally distributed test statistics, and moreover it can distinguish tests better than some other types of efficiencies like Pitman or Hodges-Lehmann (see \cite{nikitinKnjiga}).

Consider the setting of testing the null hypothesis $H_0:\theta\in\Theta_0$ against the alternative
$H_1:\theta\in\Theta_1$. Let us suppose that for a test statistic $T_n$, under $H_0$, the limit 
$\lim_{n\rightarrow\infty}P\{T_n\leq t\}=F(t)$, where $F$ is non-degenerate distribution function, exists. 
Further, suppose that $\lim_{t\rightarrow\infty}t^{-2}\log(1-F(t))=-\frac{a_T}{2}$, and
that the limit in probability $P_\theta$ $\lim_{n\rightarrow\infty}T_n=b_T(\theta)>0$, exists for $\theta \in \Theta_1$. 
The relative approximate Bahadur efficiency with respect to another test statistic $V_n$ is
\begin{equation*}
e^{\ast}_{T,V}{(\theta)}=\frac{c^{\ast}_T{(\theta)}}{c^{\ast}_V{(\theta)}},
\end{equation*}
where \begin{equation}\label{appslope}c^{\ast}_T(\theta)=a_Tb_T^2(\theta)\end{equation} 
is the approximate Bahadur slope of $T_n$. Its limit when $\theta\to 0$ is called the local approximate Bahadur efficiency.

The tests we consider may be classified into three groups according to their limiting distributions:
asymptotically normal ones; those whose asymptotic distribution coincides with the supremum of some Gaussian process; and those whose limiting distribution is an infinite linear combination of independent and identically distributed (i.i.d.) chi-squared random variables.

For the first group of tests, the coefficient $a_T$ is the inverse of the limiting variance. For the second, it is the inverse of the supremum of the covariance
function of the limiting process (see \cite{marcus1972sample}). For the third group, $a_T$ is  the inverse of the largest coefficient in the corresponding linear combination (see \cite{Zolotarev}), which is also equal to the largest eigenvalue of some integral operator. 

The goal of this paper is twofold. First, we propose two new classes of characterization based exponentiality tests. One of them is of weighted $L^2$-type, and the other, for the first time, is based on $L^{\infty}$ distance between two $V$-empirical Laplace transforms of the random variables that appear in the characterization. 

Secondly, we perform an extensive efficiency comparison. Unlike for the remaining two, for the third group of tests, the efficiencies have not been calculated so far. This is due to the fact that the largest eigenvalue in question usually cannot be obtained analytically. We overcome this problem using a recently proposed approximation procedure from \cite{bozin}.

The rest of the paper is organized as follows. In Section 2 we propose new tests and explore their asymptotic properties. In Section 3 we give a partial review of test statistics for testing exponentiality, together with their Bahadur slopes. Section 4 is devoted to the comparison of efficiencies. In Section 5 we present the powers of new tests. All proofs are given in two Appendix sections.

\section{New test statistics}\label{sec: testNewStat}
 In this section we present two new exponentiality tests based on the following characterization from \cite{desu1971}.
 
\begin{charact}[Desu (1970)]\label{desuCh}
Let $X_1$ and $X_2$ be two independent copies of a random variable $X$ with pdf $f(x)$. Then $X$ and $2\min(X_1, X_2)$ have the same distribution if and only if for some $\lambda>0$ $f(x)=\lambda e^{-\lambda x},$ for $x\geq0$.
\end{charact}

Let $X_1,X_2,...,X_n$ be a random sample from a non-negative continuous distribution. To test the null hypothesis that the sample comes from the exponential distribution $\mathcal{E}(\lambda)$, with an unknown $\lambda>0$, we examine the difference $\mathcal{L}^{(1)}_n(t)-\mathcal{L}^{(2)}_n(t)$, of V-empirical Laplace transforms of $X$ and $2\min(X_1,X_2)$. 

Clearly, if null hypothesis is true, the difference   $\mathcal{L}^{(1)}_n(t)-\mathcal{L}^{(2)}_n(t)$ will be small for each $t$. Taking  this into account we propose the following two classes of test statistics, with their large values considered significant:
\begin{align*}
M^{\mathcal{D}}_{n,a}&=\int_0^{\infty}\Big(\frac{1}{n}\sum_{i_1=1}^ne^{-tY_{i_1}}-\frac{1}{n^{2}}\sum_{i_1,i_2=1}^ne^{-t 2\min(Y_{i_1},Y_{i_2})}\Big)^2e^{-at}dt;\\
L^{\mathcal{D}}_{n,a}&=\sup_{t>0}\Big|\Big(\frac{1}{n}\sum_{i_1=1}^ne^{-tY_{i_1}}-\frac{1}{n^{2}}\sum_{i_1,i_2=1}^ne^{-t 2\min(Y_{i_1},Y_{i_2})}\Big)e^{-at}\Big|,
\end{align*}
where
$Y_i=\frac{X_i}{\bar{X}}$, $i=1,2,\ldots,n$, is the scaled sample.

The sample is scaled to make the test statistic ancillary for the parameter $\lambda$ and the purpose of the tuning parameter $a$ is to magnify different types of deviations from the null distribution.

\subsection{Asymptotic properties under $H_0$}

Notice that $M^{\mathcal{D}}_{n,a}$ is a V-statistic with estimated  parameter $\hat{\lambda}$, i.e. it can be represented in the form
$$M^{\mathcal{D}}_{n,a}=M^{\mathcal{D}}_{n,a}(\widehat{\lambda}_n)=\frac{1}{n^4}\sum_{i_1,i_2,i_3,i_4}H(X_{i_1},X_{i_2},{X_{i_3},X_{i_4}};a,\hat{\lambda}_n),$$
where $H$ is a symmetric function of its arguments, and $\hat{\lambda}_n$ is the reciprocal sample mean.

Similarly, for a fixed $t$, the expression in the absolute parenthesis of the statistics $L^{\mathcal{D}}_{n,a}$ is a V-statistics that can be represented as
\begin{align}\label{VstatL}
    \frac{1}{n^2}\sum_{i_1,i_2}\Phi(X_{i_1},X_{i_2};t,a,\hat{\lambda}_n),
\end{align}
where $\Phi$ is a symmetric function of its arguments. 

The asymptotic behaviour of $M^{\mathcal{D}}_{n,a}$ is given in the following theorem.

\begin{theorem}\label{raspodelaM}
	Let $X_1,..., X_n$ be i.i.d. with exponential distribution. Then
	\begin{equation*}
	nM^{\mathcal{D}}_{n,a}\stackrel{D}{\rightarrow}6\sum_{k=1}^\infty\delta_kW^2_k,
	\end{equation*}
	where $\delta_k, k=1,2,...,$ is the sequence of  eigenvalues of the integral operator $A$ defined by $Aq(x)=\int_{0}^{\infty}\tilde{h}_2(x,y;a)q(y)dF(y)$, with $\tilde{h}_2(x,y)=E(H(\cdot)|X_1=x,X_2=y)$ being the second projection of kernel $H(X_1,X_2,X_3,X_4;a,\lambda)$, and $W_{k},\;k=1,2,...,$ are independent standard normal variables.
\end{theorem}

The asymptotic behaviour of $L^{\mathcal{D}}_{n,a}$ is given in the following theorem.

\begin{theorem}\label{raspodelaL}\small
	Let $X_1,..., X_n$ be i.i.d. with exponential distribution. Then
	\begin{equation*}
	\sqrt{n} L^{\mathcal{D}}_{n,a}\stackrel{D}{\rightarrow}\sup_{t>0}|\eta(t)|,
	\end{equation*}
	where $\eta(t)$  is a centered Gaussian process with the  covariance function 
	\begin{equation*}\label{kov}
	K(s,t)=\frac{e^{-a(s+t)}st(4+8s+4s^2+8t+15st+6s^2t+4t^2+6st^2)}{4(1+s)(1+t)(1+s+t)(2+2s+t)(2+s+2t)(3+2s+2t)}
	\end{equation*}
\end{theorem}

\subsection{ Approximate Bahadur slope }

Let $\mathcal{G}=\{G(x;\theta),\;\theta>0\}$ with corresponding densities $\{g(x;\theta)\}$ be a family of alternative distribution functions with finite expectations, such that $G(x,\theta)=1-e^{-\lambda x}$, for some $\lambda>0$, if and only if $\theta=0$, and the regularity conditions for V-statistics with weakly degenerate kernels from \cite[Assumptions WD]{nikitinMetron} are satisfied.

The approximate local Bahadur slopes of  $M^{\mathcal{D}}_{n,a}$ and $L_{n,a}^{\mathcal{D}}$, for close alternatives, are derived in the following theorem.

\begin{theorem}\label{Slopovi}
	For the statistics $M^{\mathcal{D}}_{n,a}$ and
	$L^{\mathcal{D}}_{n,a}$ and a given alternative density $g(x,\theta)$ from $\mathcal{G},$ the local 
	Bahadur approximate slopes are given by
	\begin{enumerate}
	    \item[1)]
	\begin{equation*}
	c^*_M(\theta)=\delta_1^{-1}\int\limits_{0}^{\infty}\int\limits_{0}^{\infty}\tilde{h}_2(x,y)g'_{\theta}(x;0)g'_{\theta}(y;0)dxdy\cdot\theta^2+o(\theta^2), \theta\rightarrow0,
	\end{equation*}
		where $\delta_1$ is the largest eigenvalue of the integral operator $A$ with kernel $\tilde{h}_2$;
	
	\medskip

\item[2)]	\begin{equation*}
	c^*_L(\theta)=\frac{1}{\sup_tK(t,t)}\sup\Big(\int_{0}^{\infty}\tilde{\varphi}_1(x;t)g'_{\theta}(x;0)dx\Big)^2\cdot\theta^2+o(\theta^2), \theta\rightarrow0,
	\end{equation*}
	where  $\tilde{\varphi}_1(x;t)=E(\Phi(\cdot)|X_1=x)$ with $\Phi$ being defined in \eqref{VstatL}.
	\end{enumerate}
\end{theorem}
\begin{proof}
	See Appendix A. 
\end{proof}

To calculate the slope of $M^{\mathcal{D}}_{n,a}$, one needs to find  the largest eigenvalue $\delta_1$. Since it cannot be obtained analytically, we use the  approximation introduced in \cite{bozin}. The procedure utilizes the fact that $\delta_1$ is the limit of the sequence of the largest eigenvalues of linear operators defined by $(m+1)\times(m+1)$ matrices $M^{(m)}=||m_{i,j}^{(m)}||,\; 0\leq i\leq m, 0\leq j\leq m$, where
\begin{equation}\label{MatAppr}
m_{i,j}^{(m)}=\widetilde{h}_2\bigg(\frac{Bi}{m},\frac{Bj}{m}\bigg)\sqrt{e^{\frac{B(i)}{m}}-e^{\frac{B(i+1)}{m}}}\cdot\sqrt{e^{\frac{B(j)}{m}}-e^{\frac{B(j+1)}{m}}}\cdot \frac{1}{1-e^{-B}},
\end{equation}
when $m$ tends to infinity   and $F(B)$ approaches 1.

 \section{Other exponentiality tests -- a partial review}\label{sec: testStat}
In this section we present test statistics of some classical and some recent goodness-of-fit tests for the exponential distribution, along with their Bahadur local approximate slopes. For some of the test statistics, the Bahadur  local approximate slope (or exact slope which locally coincides with the approximate one) is available in the literature and for the others we derive them in Appendix B.

As indicated in Introduction, we classify the tests according to their asymptotic distribution. 
	The first group  contains  asymptotically normally distributed statistics.

\begin{itemize}
\item The test proposed by \cite{epps1986test} based on the expected value of the exponential density, with test statistic \begin{equation*}
		EP_n=\sqrt{48}\bigg(\frac{1}{n}\sum\limits_{j=1}^ne^{-\frac{X_j}{\overline{X}_n}}-\frac{1}{2}\bigg).
		\end{equation*}

Its approximate Bahadur slope is
\begin{equation*}
    c^*_{EP}(\theta)=3\Bigg(\int\limits_{0}^\infty \Big(4e^{-x}+x\Big)g'_{\theta}(x;0)dx\Bigg)^2\cdot\theta^2+o(\theta^2)
\end{equation*}

\item The score test for the Weibull shape parameter proposed by \cite{cox1984analysis}  
\begin{equation*} 
		CO_n=1+\frac{1}{n}\displaystyle\sum_{i=1}^n\bigg(1-\frac{X_i}{\overline{X}_n}\bigg)\log \frac{X_i}{\overline{X}_n}.
		\end{equation*}
Its approximate slope is
\begin{equation*}
    c^*_{CO}(\theta)=\frac{6}{\pi^2}\Bigg(\int\limits_{0}^\infty \Big((1-x)\log x+(1-\gamma)x\Big)g'_{\theta}(x;0)dx\Bigg)^2\cdot\theta^2+o(\theta^2)
\end{equation*}	
		
\item A test based on Gini coefficient from \cite{gail1978scale} 
\begin{equation*} 
		G^*_n=\Big|\frac{1}{2n(n-1)\overline{X}_n}\sum_{i,j=1}^n|X_i-X_j|-\frac{1}{2}\Big|.
		\end{equation*}	

The approximate slope is (see \cite{nikitin1996bahadur})
 \begin{equation*}
    c^*_{G}(\theta)=12\Bigg(\int\limits_0^\infty \Big(2e^{-x}+\frac{x}{2}\Big)g'_{\theta}(x;0)dx\Bigg)^2\theta^2+o(\theta^2).
\end{equation*}
		
\item The score test for the gamma shape parameter proposed by \cite{moran1951random} and \cite{tchirina2005bahadur}  
\begin{equation*} 
	MO_n=\Big|\gamma+\frac{1}{n}\sum_{i=1}^n\log \frac{X_i}{\overline{X}_n}\Big|.
		\end{equation*}		
Its approximate slope is (see  \cite{tchirina2005bahadur}) 
\begin{equation*}
    c^*_{MO}(\theta)=\Big(\frac{\pi^2}{6}-1\Big)^{-1}\Bigg(\int\limits_0^\infty (\log x-x)g'_{\theta}(x;0)dx\Bigg)^2\theta^2+o(\theta)
\end{equation*}	

\item Characterization based integral-type tests

Let the relation
	\begin{align}
	    \label{ChEqD}\omega_1(X_1,...,X_m)\overset{d}{=}\omega_2(X_1,...,X_p),
	\end{align}
where $X_1,\ldots,X_{\max\{m,p\}}$ are i.i.d. random variables, characterize the exponential distribution. Then the following types of test statistics have been proposed:

\begin{equation*}
    I_n=\int_{0}^{\infty}\Big(H_n^{(\omega_1)}(t)-H_n^{(\omega_2)}(t)\Big)dF_n(t),
\end{equation*}
where $H_n^{(\omega_1)}(t)$ and $H_n^{(\omega_2)}(t)$ are $V$-empirical distribution functions of $\omega_1$ and $\omega_2$, respectively,  and $F_n$ is the  empirical distribution function, and 

\begin{equation}\label{laplas}
    J_{n,a}=\int_{0}^{\infty}\Big(L_n^{(\omega_1)}(t)-L_n^{(\omega_2)}(t)\Big)e^{-at}dt,
\end{equation}
where $L_n^{(\omega_1)}(t)$ and $L_n^{(\omega_2)}(t)$ are $V$-empirical Laplace transforms of  $\omega_1$ and $\omega_2$, respectively, applied to the scaled sample, and $a>0$ is the tuning parameter.

From these groups of tests we take the following representatives

\begin{itemize}
    \item $I_{n,k}^{(1)}$, proposed in \cite{jovanovic}, based on the Arnold and Villasenor characterization, where $\omega_1(X_1,...,X_k)=\max(X_1,...,X_k)$ and $\omega_2(X_1,...,X_k)=X_1+\frac{X_2}{2}+\cdots \frac{X_k}{k}$ (see \cite{arnold2013exponential}, \cite{milovsevic2016some});
    \item $I_{n}^{(2)}$, proposed in \cite{Publ}, based on the Milo\v sevi\'c-Obradovi\'c characterization,  where $\omega_1(X_1,X_2)=\max(X_1.X_2)$ and $\omega_2(X_1,X_2,X_3)=\min(X_1,X_2)+X_3$ (see \cite{milovsevic2016some});
    \item $I_n^{(3)}$, proposed in \cite{bojanaMetrika}, based on the Obradovi\'c characterization, where $\omega_1(X_1,X_2,X_3)=\max(X_1,X_2,X_3)$ and $\omega_2(X_1,X_2,X_3,X_4)=X_1+{\rm med}(X_2,X_3,X_4)$ (see \cite{obradovic2014three});
    \item $I_n^{(4)}$, proposed in \cite{volkova2015goodness}, based on the Yanev-Chakraborty characterization, where $\omega_1(X_1,X_2,X_3)=\max(X_1,X_2,X_3)$ and $\omega_2(X_1,X_2,X_3)=\frac{X_1}{3}+\max(X_2,X_3)$ (see \cite{yanev2013characterizations});
    \item $I_n^{\mathcal{D}}$ based on the Desu characterization \ref{desuCh} characterization;
     \item $I_n^{\mathcal{P}}$ based on the Puri-Rubin characterization, where $\omega_1(X_1)=X_1$ and $\omega_2(X_1,X_2)=|X_1-X_2|$ (see \cite{puri1970characterization});
    \item $J_{n,a}^{\mathcal{D}}$, proposed in \cite{MilosevicObradovicPapers}, based on the Desu Characterization \ref{desuCh};
    \item $J_{n,a}^{\mathcal{P}}$, proposed in \cite{MilosevicObradovicPapers}, based on the Puri-Rubin characterization.
\end{itemize}

Since these statistics are very similar, we give  general expressions for their Bahadur approximate slopes. 

Statistics $I_n$ are non-degenerate V-statistics with some kernel $\Psi$ and their approximate slope is (see \cite{nikitinMetron})
\begin{align}\label{integralniNagib}
    c^{*}_I(\theta)=\frac{1}{\rm Var\psi(X_1)}\Big(\int \psi(x)g'_{\theta}(x;0)dx\Big)^2 \cdot\theta^2 + o(\theta^2),
\end{align}
where $\psi(x)=E\Psi(\cdot |X_1=x)$.

Statistics $J_{n,a}$ are, due to the sample scaling, non-degenerate V-statistics with estimated parameters. Nevertheless, the formula \eqref{integralniNagib} is applicable here also, with $\Psi$ being the kernel of the test statistic as if the scaling were done using the real value of $\lambda$ (see \cite{MilosevicObradovicPapers} for details).
		\end{itemize}

The second group contains statistics whose limiting distribution is the supremum of some centered Gaussian process.

\begin{itemize}
	\item Lilliefors modification of the Kolmogorov-Smirnov test \begin{equation*}
		KS_n=\sup|F_n(t)-(1-e^{-\frac{t}{\bar{X}}})|.
		\end{equation*}
	The approximate slope is  (see \cite{nikitin2007lilliefors})
	\begin{equation*}
    c^*_{KS}(\theta)=\frac{1}{\sup\limits_{x\geq0}(e^{-2x}(e^x-x^2-1))}\sup\limits_{x\geq0}\Big(xe^x\int_0^\infty G'_{\theta}(u;0)du-G(x;0)dx\Big)^2\cdot \theta^2+o(\theta^2).
\end{equation*}	

	\item Characterization based supremum-type tests 
	
	Using the characterizations of the type \eqref{ChEqD}, another proposed type of test statistics is
	
\begin{equation*}
    D_n=\sup_{t>0}\Big|H_n^{(\omega_1)}(t)-H_n^{(\omega_2)}(t)\Big|.
\end{equation*}	  

    From this group of tests we take the following representatives:
    
 $D_{n,k}^{(1)}$, proposed in \cite{jovanovic};
    $D_{n}^{(2)}$, proposed in \cite{Publ};
   $D_n^{(3)}$, proposed in \cite{bojanaMetrika};
     $D_n^{(4)}$, proposed in \cite{volkova2015goodness}, based on the same characterizations as for the respective integral-type statistics, $D_{n}^{\mathcal{D}}$ based on Desu characterization \ref{desuCh} and  $D_{n}^{\mathcal{P}}$ based on Puri-Rubin characterization (\cite{puri1970characterization}).

    Statistics from this group are asymptotically distributed as a supremum of some non-degenerate V-empirical processes, and the expression in the absolute parenthesis, for a fixed $t$ is a V-statistic with some kernel $\Psi(X_1,\ldots,X_{\max\{m,p\}};t)$. Their approximate slopes is (see \cite{nikitin2010large})
  \begin{align*}
    c^{\star}_D(\theta)=\frac{1}{\sup_{t>0}{\rm Var}\psi(X_1;t)}\sup_{t>0}\Big(\int \psi(x;t)g'_{\theta}(x;0)dx\Big)^2 \cdot\theta^2 + o(\theta^2),
\end{align*}
where $\psi(x;t)=E\Psi(\cdot;t |X_1=x)$.

\end{itemize}

The third group contains statistics whose limiting distribution is an infinite linear combination of i.i.d. chi-squared random variables. Each of the presented statistics, except the last one, is of the form 
\begin{align*}
    T_n=\int_0^\infty U_n^2(t;\hat{\mu})w(t)dt,
\end{align*}
where $U_n(t;\hat{\mu})$ is an empirical process of order 1 with estimated parameter. It also can be viewed as a weakly degenerate V-statistics with estimated parameters, with
some kernel $\Phi(X_1,X_2;\hat{\mu})$, where $\mu=E_{\theta}X_1$. Then, the Bahadur approximate slope of such statistic is

\begin{align}\label{BAS_L2}
    c^*(\theta)&=(2\delta_\Phi)^{-1}\int\limits_0^\infty\int\limits_0^\infty\Bigg(2\Phi(x,y;1)g'_{\theta}(x;0)g'_{\theta}(y;0)+4\Phi'_{\mu(\theta)}(x,y;1)\mu'_{\theta}(0)g'(x;0)g(y;0)\\&+\Phi''_{\mu^2(\theta)}(x,y;1)(\mu'_\theta(0))^2g(x;0)g(y;0)\Bigg)dxdy \cdot \theta^2+o(\theta^2)\nonumber,
\end{align}
where $\delta_T$ is the largest eigenvalue of the  integral operator $Aq(s)\int_0^\infty K(s,t)w(t)q(t)dt$, where $K(s,t)=\lim_{n\to\infty}n{\rm Cov}(U_n(t),U_n(s))$ is the limiting covariance function.

Hence it suffices to present only the kernels and limiting covariance functions  of test statistics. We consider the following tests:

\begin{itemize}
    \item Lilliefors modification of the Cramer-von Mises test
    \begin{equation*} 
    \omega^2_n=\int_{0}^{\infty}(F_n(t)-(1-e^{-\frac{t}{\bar{X}}}))^2\frac{1}{\bar{X}}e^{-\frac{t}{\bar{X}}}dt.
		\end{equation*}

		Its kernel is
				\begin{equation*}
		    \Phi_{\omega^2}(x,y;\mu(\theta))=e^{-\max(\frac{x}{\mu(\theta)},\frac{y}{\mu(\theta)})}-e^{-\frac{x}{\mu{\theta}}}-e^{-\frac{y}{\mu(\theta)}}+\frac{1}{2}(e^{-2\frac{x}{\mu(\theta)}}+e^{-2\frac{y}{\mu(\theta)}})+\frac{1}{3};
		\end{equation*}
		and the covariance function is
	\begin{align*}
	    K_{\omega^2}(s,t)&=e^{-\frac{3}{2}s-\frac{3}{2}t}(e^{\min(s,t)}-1-st).
	\end{align*}
		
	\item Lilliefors modification of the Anderson-Darling	test 
	\begin{equation*} 
	{\rm AD}_n=\int_{0}^{\infty}\frac{(F_n(t)-(1-e^{-\frac{t}{\bar{X}}}))^2}{\bar{X}(1-e^{-\frac{t}{\bar{X}}})}dt.
		\end{equation*}
		
Its kernel is

		\begin{equation*}
		    \Phi_{\rm AD}(x,y;\mu(\theta))=\frac{x}{\mu(\theta)}+\frac{y}{\mu(\theta)}-1-\log(e^{\max(\frac{x}{\mu(\theta)},\frac{y}{\mu(\theta)})}-1).
		\end{equation*}
	and the covariance function is
	\begin{align*}
	    K_{AD}(s,t)&=\frac{e^{-s-t}(e^{\min(s,t)}-1-st)}{\sqrt{(1-e^{-s})(1-e^{-t})}}.
	\end{align*}

			\item  A test proposed by \cite{baringhaus1991class}
		\begin{equation*}
		{\rm BH}_n=\int_0^\infty\Big((1+t)\psi'_n(t)+\psi_n(t)\Big)^2e^{-at}dt,
		\end{equation*}
		where $\psi_n(t)$ is the empirical Laplace transform. Its kernel is
		\begin{equation*}
		    \Phi_{\rm BH}(x,y;\mu(\theta))=\frac{(1-x)(1-y)}{x+y+a\mu(\theta)}-\frac{x\mu(\theta)+y\mu(\theta)-2xy}{(a\mu(\theta)+x+y)^2}+\frac{2xy\mu(\theta)}{(a\mu(\theta)+x+y)^3};
		\end{equation*}
			and the covariance function is
	\begin{align*}
	    K_{BH}(s,t)&=\frac{1+s+t+2st}{(1+s+t)^3}-\frac{1}{(1+s)^2(1+t)^2}.
	\end{align*}
	
	
		\item The test  proposed by \cite{henze1992new}
		\begin{equation*}
		{\rm HE}_n=\int_0^\infty\Big(\psi_n(t)-\frac{1}{1+t}\Big)^2e^{-at}dt;
		\end{equation*}
		
		Its kernel is
		\begin{align*}
		    \Phi_{\rm HE}(x,y;\mu(\theta))&=1+\frac{\mu(\theta)}{a\mu(\theta)+x+y}+ae^aEi(-a)+e^{a+\frac{x}{\mu(\theta)}}Ei(-(a+\frac{x}{\mu(\theta)}))\\&+e^{a+\frac{y}{\mu(\theta)}}Ei(-(a+\frac{y}{\mu(\theta)}));
		\end{align*}
and the covariance function is
	\begin{align*}
	    K_{HE}(s,t)&=\frac{s^2t^2}{(s+t+1)(s+1)^2(t+1)^2}.
	\end{align*}
		
		\item The test proposed by \cite{henze2002tests}
		\begin{equation*}
		W_n=\int_0^\infty\Big(\psi_n(t)-\frac{1}{1+t}\Big)^2(1+t)^2e^{-at}dt.
		\end{equation*}
		Its kernel is
		\begin{align*}
		    \Phi_{W}(x,y;\mu(\theta))&=\frac{1}{a}-\frac{\mu(\theta)(\mu(\theta)(1+a)+x)}{(a\mu(\theta)+x)^2}- \frac{\mu(\theta)(\mu(\theta)(1+a)+y)}{(a\mu(\theta)+y)^2}\\&+\frac{2\mu^3(\theta)}{(a\mu(\theta)+x+y)^3}+\frac{2\mu^2(\theta)}{(a\mu(\theta)+x+y)^2}+\frac{\mu(\theta)}{a\mu(\theta)+x+y};
		\end{align*}
	and the covariance function is
	\begin{align*}
	    K_{W}(s,t)&=\frac{s^2t^2}{(s+t+1)(s+1)(t+1)}.
	\end{align*}
		
		\item Two tests proposed by \cite{henze2002goodness}
		\begin{equation*}
		{\rm HM}_n=\int_0^\infty(s_n(t)-tc_n(t))^2\omega_{i}(t)dt, \;\;i=1,2,
		\end{equation*}
		where $\omega_1(t)=e^{-at}$ i $\omega_2(t)=e^{-at^2}$. Their kernels are
		\begin{equation*}
		\begin{split}
    \Phi_{\rm HM1}(x,y;\mu(\theta))=&\frac{a\mu^2(\theta)}{2(a^2\mu^2(\theta)+(x-y)^2)}-\frac{a\mu^2(\theta)}{2(a^2\mu^2(\theta)+ (x+y)^2)}\\ &+a\frac{a^2\mu^2(\theta)-3(x-y)^2}{(a^2\mu^2(\theta)+(x-y)^2)^3}+a\frac{a^2\mu^2(\theta)-3(x+y)^2}{(a^2\mu^2(\theta)+(x+y)^2)^3} \\&- \frac{2a\mu^3(\theta)(x + y)}{(a^2\mu^2(\theta) +(x + y)^2)^2},
    \end{split}
\end{equation*}
and 
		\begin{equation*}
		    \begin{split}
    \Phi_{\rm HM2}(x,y;\mu(\theta))=\frac{\sqrt{\pi}}{4\sqrt{a}}&\Bigg(\bigg(\frac{1}{2a}-\frac{x+y}{a\mu(\theta)}-\frac{(x-y)^2}{4a^2\mu^2(\theta)}-1\bigg)e^{-\frac{(x+y)^2}{4a\mu^2(\theta)}}\\&+\bigg(1+\frac{1}{2a}-\frac{(x-y)^2}{4a^2\mu^2(\theta)}\bigg)e^{-\frac{(x-y)^2}{4a\mu^2(\theta)}}\Bigg);
    \end{split}
		\end{equation*}
		and the covariance function is
	\begin{align*}
	    K_{HM}(s,t)&=\frac{st(s^2+t^2+1)}{(1+(s-t)^2)(1+(s+t)^2)}-\frac{st}{(1+s^2)(1+t^2)}.
	\end{align*}
	
		\item Characterization based $L^2$-type test proposed by \cite{cuparic2018new}.

		\begin{equation}\label{laplasSort}
    M^{\mathcal{P}}_{n,a}=\int_{0}^{\infty}\Big(L_n^{(1)}(t)-L_n^{(2)}(t)\Big)^2e^{-at}dt.
\end{equation}
Its slope is 
\begin{equation*}
		\begin{aligned}
		    c^*_{M}(\theta)&=(2\delta_1)^{-1}\int\limits_{0}^{\infty}\int\limits_{0}^{\infty}\bigg[\frac{1}{6}e^{a-x-y}\text{Ei}(-a) \Big(a (e^x-2)(e^y-2)-e^x-e^y+4\Big)\\& +\frac{1}{6}e^{-a-x-y}\Big(\text{Ei}(a) (4 a+e^x+e^y-4)-(\text{Ei}(a+x) (4 (a+x-1)+e^y)\\&+\text{Ei}(a+y) (4 (a+y-1)+e^x)-4 (a+x+y-1) \text{Ei}(a+x+y))\Big)-\frac{1}{2}\\&+\frac{1}{3}(e^{-x}+e^{-y})+\frac{1}{6(a+x+y)}\bigg]g'(x;0)g'(y;0)dxdy\cdot\theta^2+o(\theta^2).
		\end{aligned}
		\end{equation*}
	\end{itemize}

\section{Comparison of efficiencies}

In this section we calculate approximate  local relative Bahadur efficiencies  of test statistics introduced in Sections \ref{sec: testNewStat} and \ref{sec: testStat} with respect to the likelihood ratio test (see \cite{bahadur1967rates}). Likelihood ratio tests are known to have optimal Bahadur efficiencies and they are therefore used as benchmark for comparison.

The alternatives we consider are the following:
\begin{itemize}
		\item a Weibull distribution with  density
		\begin{equation}\label{weibull}
		g(x,\theta)=e^{-x^{1+\theta}}(1+\theta)x^\theta,\theta>0,x\geq0;
		\end{equation}
		\item a gamma distribution with density
		\begin{equation}\label{gamma}
		g(x,\theta)=\frac{x^\theta e^{-x}}{\Gamma(\theta+1)},\theta>0,x\geq0;
		\end{equation}
		\item a linear failure rate (LFR) distribution with density 
		\begin{equation}\label{lfr}
		g(x,\theta)=e^{-x-\theta\frac{x^2}{2}}(1+\theta x),\theta>0,x\geq0;
		\end{equation}
		\item a mixture of exponential distributions with negative weights (EMNW($\beta$)) with density
		\begin{equation*}
		g(x,\theta)=(1+\theta)e^{-x}-\theta\beta e^{-\beta x},\theta\in\left(0,\frac{1}{\beta-1}\right],x\geq0;
		\end{equation*}
	\end{itemize}

	\begin{figure}[!ht]
	\begin{center}
	\includegraphics[scale=0.60]{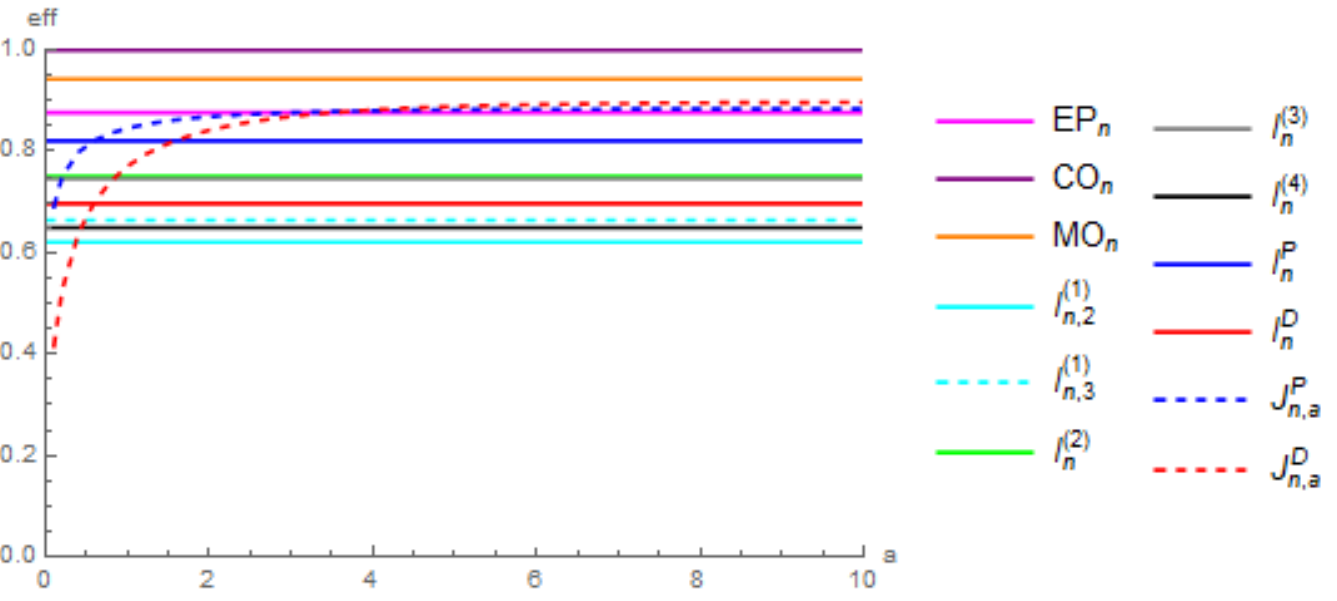}
	\includegraphics[scale=0.55]{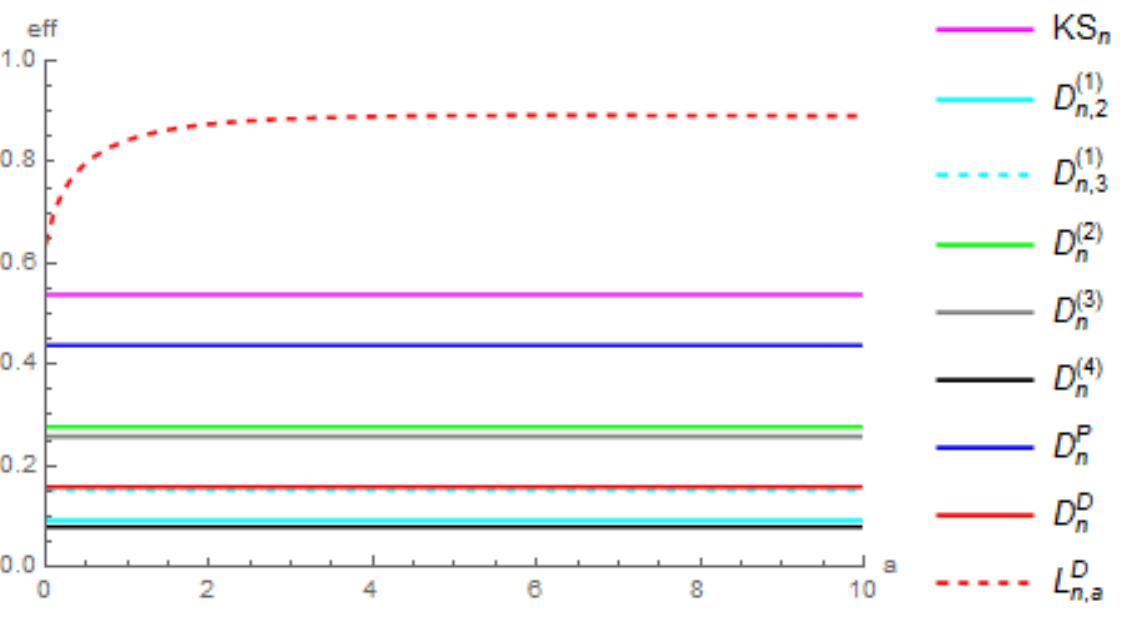}
	\includegraphics[scale=0.55]{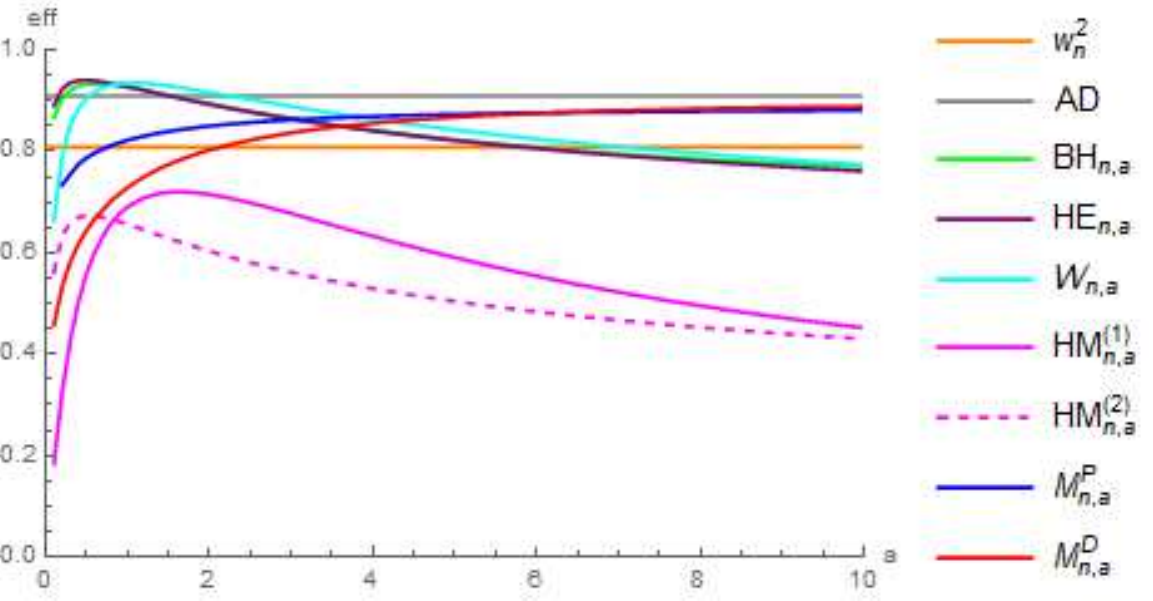}
		\caption{Local approximate Bahadur efficiencies w.r.t. LRT for a Weibull alternative}
		\label{fig: appslopeV1}
	\end{center}
	
\end{figure}


	\begin{figure}[!ht]
	\begin{center}
	\includegraphics[scale=0.60]{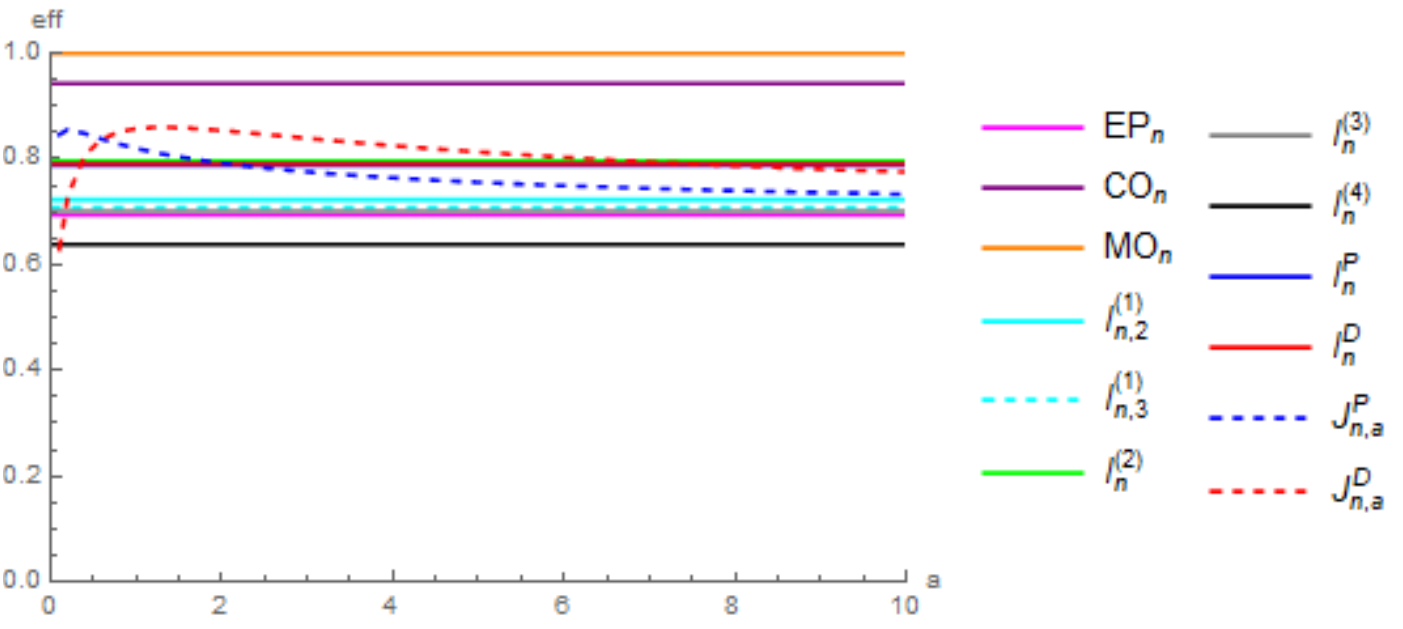}
	\includegraphics[scale=0.55]{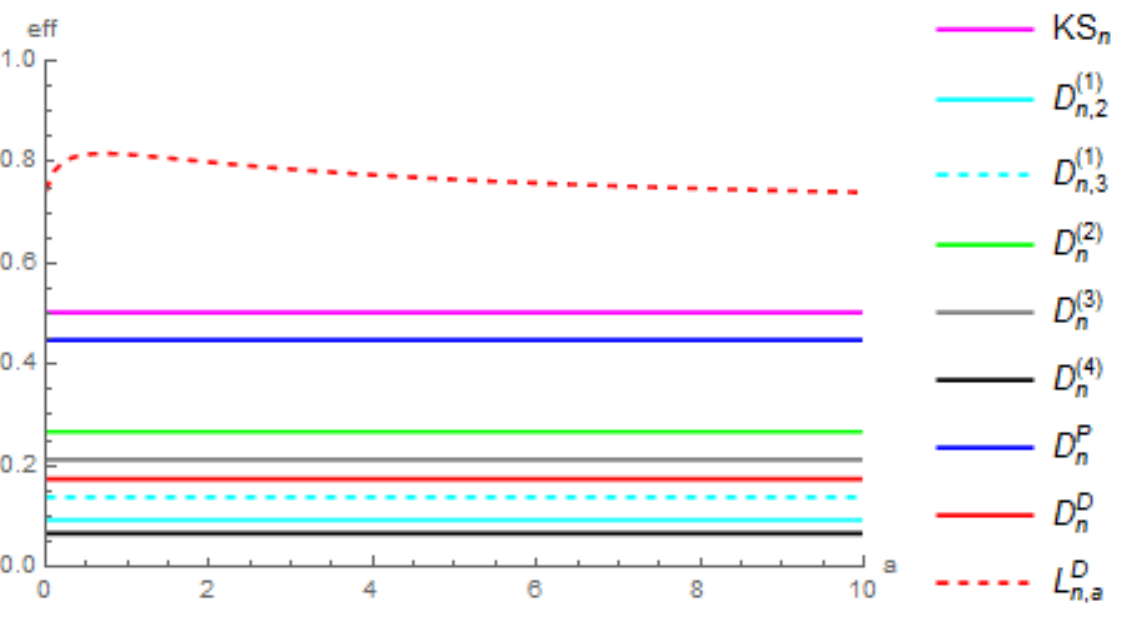}
	\includegraphics[scale=0.55]{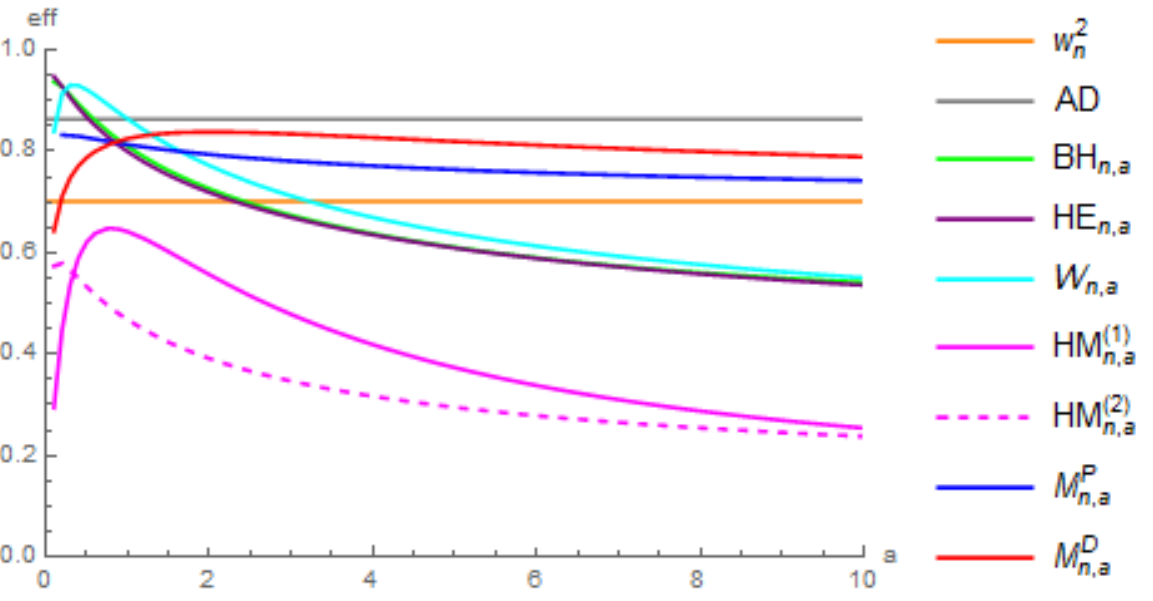}
		\caption{Local approximate Bahadur efficiencies w.r.t. LRT for a Gamma alternative}
		\label{fig: appslopeG1}
	\end{center}
	\end{figure}

	\begin{figure}[!ht]
	\begin{center}
	\includegraphics[scale=0.60]{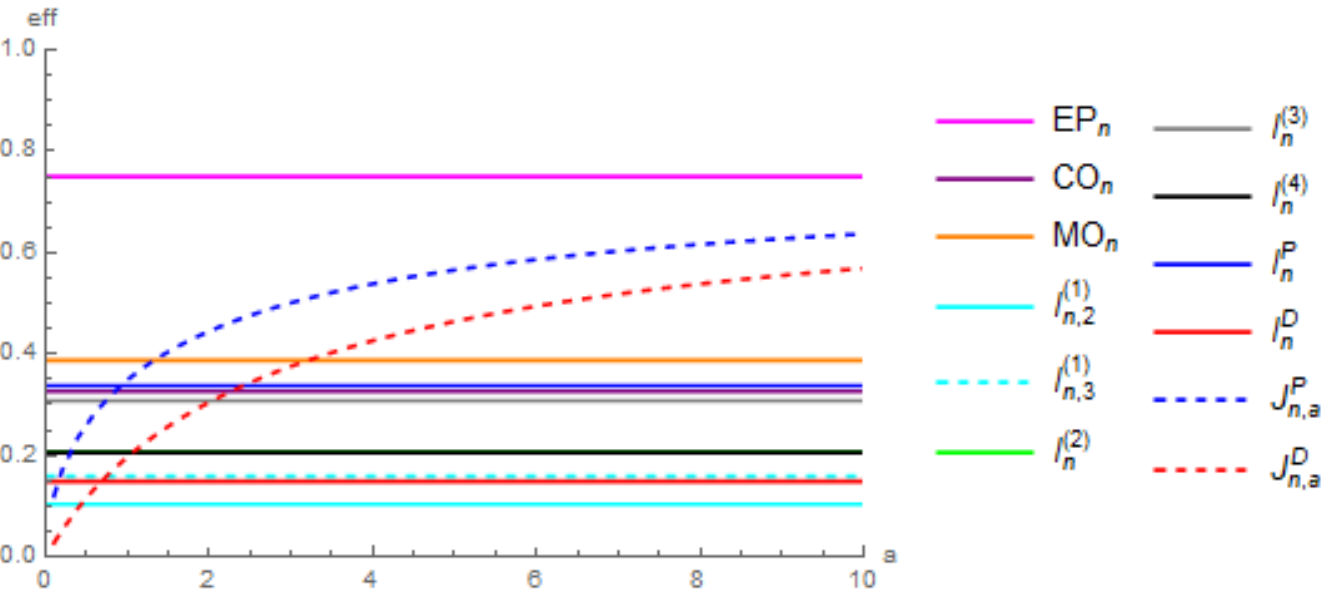}
	\includegraphics[scale=0.55]{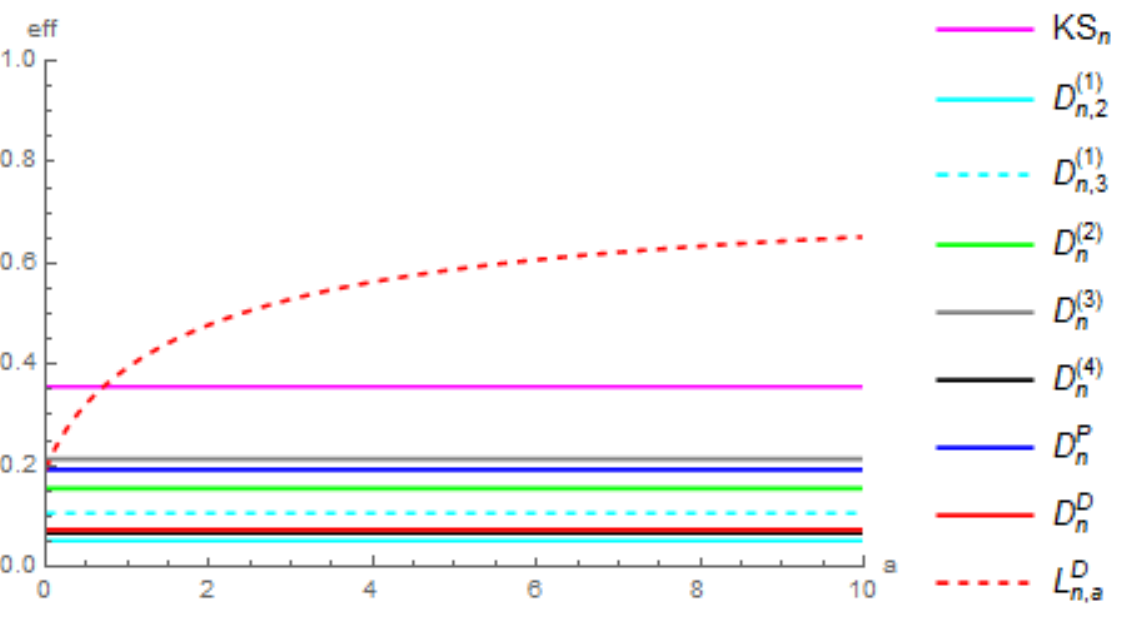}
	\includegraphics[scale=0.55]{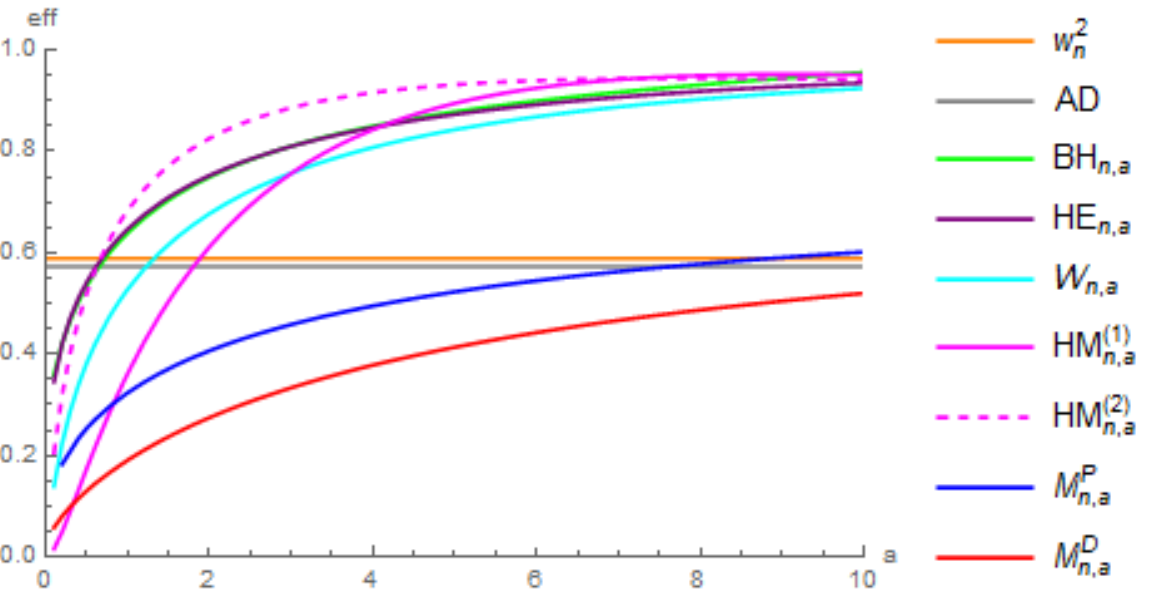}
		\caption{Local approximate Bahadur efficiencies w.r.t. LRT for a LFR alternative}
		\label{fig: appslopeL1}
	\end{center}
\end{figure}

\begin{figure}[!ht]
	\begin{center}
	\includegraphics[scale=0.60]{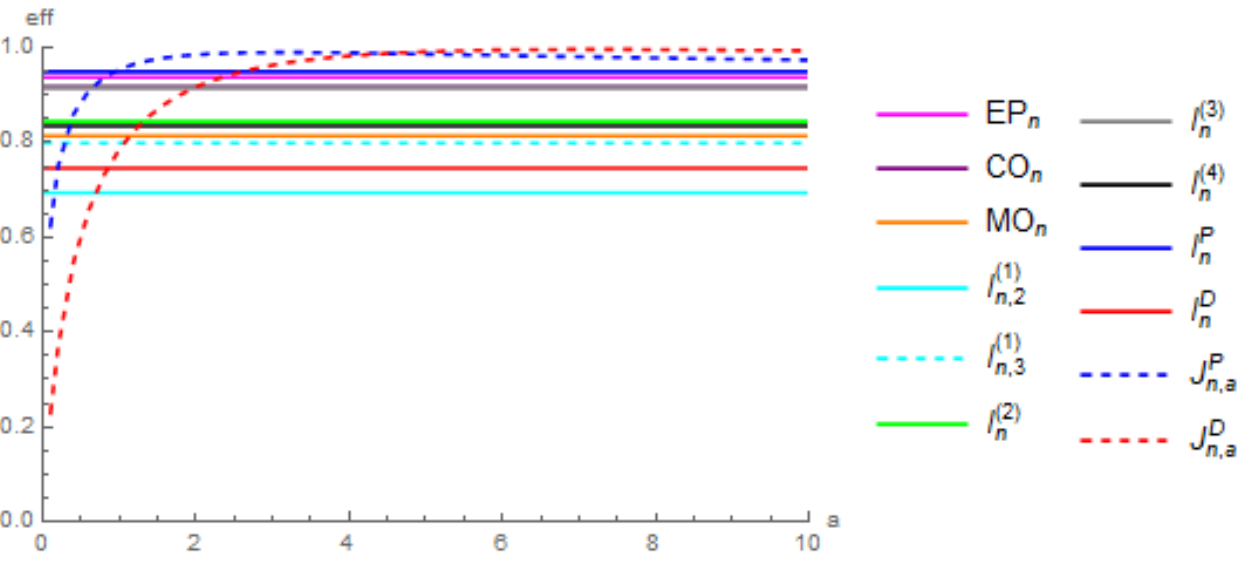}
	\includegraphics[scale=0.55]{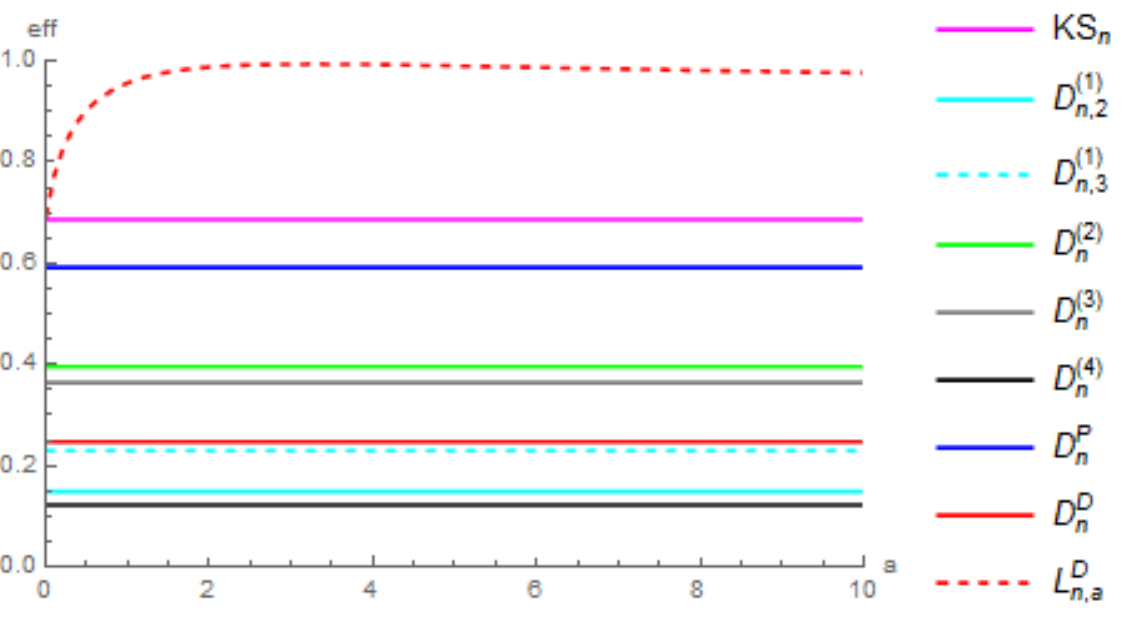}
	\includegraphics[scale=0.55]{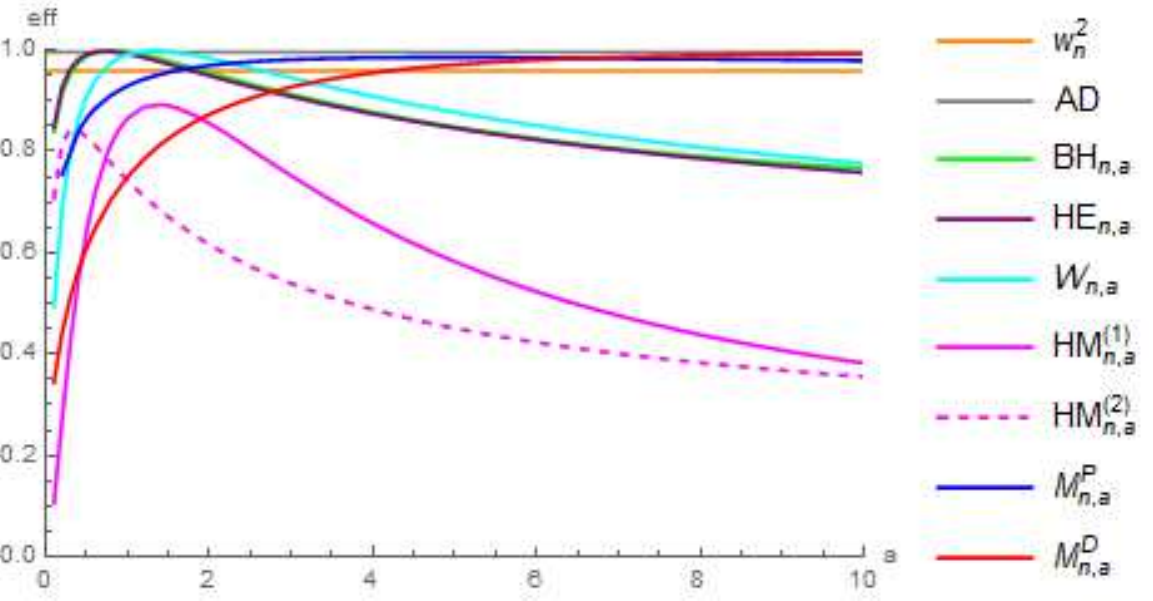}
		\caption{Local approximate Bahadur efficiencies w.r.t. LRT for a EMNW(3) alternative}
		\label{fig: appslopeM1}
	\end{center}
\end{figure}

On Figures \ref{fig: appslopeV1}-\ref{fig: appslopeM1}, there are plots of local approximate Bahadur efficiencies as a function of the tuning parameter. For tests with no such parameter straight lines are drawn. To avoid too many lines on the same plot, there are three separate plots given for each alternative, each corresponding to one of the classes of tests from Section \ref{sec: testStat}.

As a rule we can notice that, in the class of supremum-type statistics, new test $L^{\mathcal{D}}_{n,a}$ is by far the most efficient. On the other hand, supremum-type test based on characterizations that use U-empirical distribution functions, are the least efficient among all considered tests.

The impact of the tuning parameter, for the tests that have got it,  is also visible in all the figures. It is interesting to note that this impact is different for various tests in terms of monotonicity of the plotted functions.

Another general conclusion is that the ordering of the tests depends on the alternative and that there is no most efficient test to be recommended in any situation.

The CO and MO tests are known to be locally optimal for Weibull and gamma alternatives, respectively, so they are the most efficient in these cases.
However, there are quite a few other tests that perform very well in there cases.
In the case of the LFR alternative, the most efficient are EP and HM$^{(1)}_{n,a}$ and HM$^{(2)}_{n,a}$. It it interesting that for other alternatives the latter two tests are among the least efficient.

In the case of the EMNW alternative, the integral and supremum-type tests based on the characterizations via Laplace transforms, as well as most of the $L^2$ test reach, for some value of the tuning parameter, an efficiency close to one.

\section{Powers of new tests}
In this section we present the simulated powers of our new tests against different alternatives. The list of alternatives is chosen to be in concordance with the papers with extensive power comparison studies. The alternatives are:
\begin{itemize}
    \item a Weibull $W(\theta)$ distribution with density \eqref{weibull};
    \item a gamma $\Gamma(\theta)$ distribution with density \eqref{gamma};
    \item a half-normal $HN$ distribution with density 
    \begin{equation*}
        g(x)=\sqrt{\frac{2}{\pi}}e^{-\frac{x^2}{2}}, x\geq0;
    \end{equation*}
    \item a uniform $U$ distribution with density 
    \begin{equation*}
        g(x)=1, 0\leq x\leq 1;
    \end{equation*}
    \item a Chen's $CH(\theta)$ distribution with density
    \begin{equation*}
        g(x,\theta)=2\theta x^{\theta-1}e^{x^\theta-2(1-e^{x^\theta})}, x\geq 0;
    \end{equation*}
    \item a linear failure rate $LF(\theta)$ distribution  with density \eqref{lfr};
    \item a modified extreme value $EV(\theta)$ distributions with density
    
    \begin{equation*}
        g(x,\theta)=\frac{1}{\theta}e^{\frac{1-e^x}{\theta}+x}, x\geq 0;
    \end{equation*}
    \item a log-normal $LN(\theta)$ distribution with density
       \begin{equation*}
        g(x,\theta)=\frac{1}{x\sqrt{2\pi\theta^2}}e^{-\frac{(\log x)^2}{2\theta^2}}, x\geq 0;
    \end{equation*}
    \item a Dhillon $DL(\theta)$ distribution with density  
    \begin{equation*}
        g(x,\theta)=\frac{\theta+1}{x+1}(\log(x+1))^\theta e^{-(\log(x+1))^{\theta+1}}, x\geq 0.
    \end{equation*}
\end{itemize}

The powers, for aforementioned alternatives, and different choices of the tuning parameter are estimated using the Monte Carlo procedure with 10000 replicates at the level of significance 0.05.

The results are presented in Tables \ref{fig: comparison20} and \ref{fig: comparison50}. In addition, we provide the bootstrap expected power estimate for data-driven optimal value of the tuning parameter (see \cite{allison} for details). Some steps to overcome "random nature" of selected parameters are made in \cite{tenreiro2019automatic}, but some questions still remain open and are planned for future research. 

\begin{table}[htbp]
	\centering
	\caption{Percentage of rejected hypotheses for  $n=20$ }
	\resizebox{\textwidth}{!}{
		\begin{tabular}{cccccccccccccccccc}
			\rotatebox[origin=c]{90}{Alt.} & \rotatebox[origin=c]{90}{$Exp(1)$}& \rotatebox[origin=c]{90}{$W(1.4)$}  & \rotatebox[origin=c]{90}{$\Gamma(2)$} & \rotatebox[origin=c]{90}{$HN$} & \rotatebox[origin=c]{90}{$U$} & \rotatebox[origin=c]{90}{$CH(0.5)$} & \rotatebox[origin=c]{90}{$CH(1)$} & \rotatebox[origin=c]{90}{$CH(1.5)$} & \rotatebox[origin=c]{90}{$LF(2)$} & \rotatebox[origin=c]{90}{$LF(4)$} &  \rotatebox[origin=c]{90}{$EV(1.5)$}& \rotatebox[origin=c]{90}{$LN(0.8)$} & \rotatebox[origin=c]{90}{$LN(1.5)$} & \rotatebox[origin=c]{90}{$DL(1)$} & \rotatebox[origin=c]{90}{$DL(1.5)$} & \rotatebox[origin=c]{90}{$W(0.8)$}  & \rotatebox[origin=c]{90}{$\Gamma(0.4)$}\\\hline
			$M^{\mathcal{D}}_{n,0.2}$ & 5 & 24 &45 &9 & 20 & 7 & 8 & 7 & 13 & 20 & 18 & 58 & 12 & 28 & 63 & 18 & 85 \\
			$M^{\mathcal{D}}_{n,0.5}$ & 5 & 34 & 55 & 15 & 33 & 11 & 11 & 11 & 20 & 29 & 26 &  61 & 14 & 36 & 76 & 15 & 81\\
			$M^{\mathcal{D}}_{n,1}$  & 5 & 43 &	63 & 21 & 46 & 15 & 15 & 15 & 27 & 39 & 35 &  60 & 18 & 38 & 80 & 12 & 77\\
			$M^{\mathcal{D}}_{n,2}$  & 5 & 44 &	64 & 23 & 56 & 17 &	16 & 17 & 31 & 44 & 42 & 56 & 24 & 38 & 83 & 11 & 74\\
			$M^{\mathcal{D}}_{n,5}$  & 5 &  49 &	64 & 29 & 67 & 20 &	21 & 21 & 37 & 53 & 52 & 46 & 33 & 35 & 80 & 10 & 70\\ 
			$M^{\mathcal{D}}_{n,10}$ & 5 & 47 & 63 & 30 & 73 & 21 & 21 & 21 & 38 & 54 & 54 & 41 & 41 & 31 & 79 & 10 & 66\\
			$M^{\mathcal{D}}_{n,\widehat{a}}$ & 5 & 44 & 61 & 27 & 75 & 19 & 20 & 19 & 36 & 50 & 53 & 58 & 40 & 35 & 79 & 13 & 78\\
			
			$L^{\mathcal{D}}_{n,0.2}$ & 5 & 41 & 61 & 21 & 53. & 15 & 15 & 15 & 28 & 42 & 39 & 58 & 18 & 38 & 80 & 11 & 73\\
			$L^{\mathcal{D}}_{n,0.5}$ & 5 & 44 & 64 & 24 & 61 & 18 & 17 & 17 & 33 & 47 & 44 & 54 & 23 & 15 & 56 & 10 & 71\\
			$L^{\mathcal{D}}_{n,1}$ & 5 & 47 & 64 & 27 & 66 & 19 & 20 & 19 & 35 & 50 & 50 & 49 & 30 & 35 & 81 & 10 & 69\\
            $L^{\mathcal{D}}_{n,2}$ & 5 & 48 & 64 & 29 & 71 & 21 & 21 & 21 & 37 & 53 & 52 & 42 & 37 & 62 &  95 & 10 & 68  \\
            $L^{\mathcal{D}}_{n,5}$ & 5 &  50 & 64 &  32 & 78 & 23 &  23 & 23 &  41 &  57 &  58 & 40 & 48 &  86 &  99 & 12 & 65 \\
            $L^{\mathcal{D}}_{n,10}$ &5 &  50 & 61 &  31 & 77 & 23 & 21 & 22 &  39 & 53 & 22 & 35 & 49 & 29 & 77 & 11 & 62 \\
            $L^{\mathcal{D}}_{n,\widehat{a}}$ & 5 & 46 & 61 & 27 & 73 & 20 & 20 & 20 & 36 & 51 & 53 & 55 & 39 & 35 & 78 & 11 & 72\\
		  \end{tabular}
	}
	\label{fig: comparison20}
\end{table}
\begin{table}[htbp]
	\centering
	\caption{Percentage of rejected hypotheses for $n=50$}
	\resizebox{\textwidth}{!}{
		\begin{tabular}{cccccccccccccccccc}
			\rotatebox[origin=c]{90}{Alt.} & \rotatebox[origin=c]{90}{$Exp(1)$}& \rotatebox[origin=c]{90}{$W(1.4)$}  & \rotatebox[origin=c]{90}{$\Gamma(2)$} & \rotatebox[origin=c]{90}{$HN$} & \rotatebox[origin=c]{90}{$U$} & \rotatebox[origin=c]{90}{$CH(0.5)$} & \rotatebox[origin=c]{90}{$CH(1)$} & \rotatebox[origin=c]{90}{$CH(1.5)$} & \rotatebox[origin=c]{90}{$LF(2)$} & \rotatebox[origin=c]{90}{$LF(4)$} &  \rotatebox[origin=c]{90}{$EV(1.5)$}& \rotatebox[origin=c]{90}{$LN(0.8)$} & \rotatebox[origin=c]{90}{$LN(1.5)$} & \rotatebox[origin=c]{90}{$DL(1)$} & \rotatebox[origin=c]{90}{$DL(1.5)$} & \rotatebox[origin=c]{90}{$W(0.8)$}  & \rotatebox[origin=c]{90}{$\Gamma(0.4)$}\\\hline
			$M^{\mathcal{D}}_{n,0.2}$ & 5 & 57 & 89 & 17 & 41 & 12 & 10 & 10 & 26 & 41 & 33 & 98 & 30 & 74 & 98 & 36 & 100\\
		 $M^{\mathcal{D}}_{n,0.5}$ & 5 & 70 & 93 & 25& 65 & 16 & 16 & 17 & 38& 59 & 52 &  98 & 42 &  77 &  99 & 35 &  99\\
			$M^{\mathcal{D}}_{n,1}$  & 5 & 76 &  95 & 35 & 79 & 21 & 21 & 22 & 49 & 70 & 64 &  97 & 55 &  76 &  100 & 34 &  99\\
			$M^{\mathcal{D}}_{n,2}$ & 5 & 81 &  96 & 43 & 90 & 29 & 29 & 29 & 59 & 79 &	76 & 92 & 66 & 73 &  100 & 35 &  99\\
			$M^{\mathcal{D}}_{n,5}$  & 5 & 83 &  96 & 51 & 96 & 35 & 35 & 35 & 67 & 86 & 85 & 82 & 81 & 66 & 100 & 36 &  99\\
			$M^{\mathcal{D}}_{n,10}$ & 5 & 86 & 96 & 58 & 98 & 41 & 41 & 41 & 73 & 89 & 90 & 73 & 86 & 60 & 99 & 38 & 99\\ 
			$M^{\mathcal{D}}_{n,\widehat{a}}$ & 5& 81 & 95 & 56 & 96 & 41 & 41 & 40 & 71 & 87 & 90 & 97 & 87 & 74 & 99 & 36 & 99\\
			
		$L^{\mathcal{D}}_{n,0.2}$ & 5 & 79 & 96 & 41 & 90 & 26 & 27 & 26 & 58 & 78 & 74 & 96 & 62 & 77 & 100 & 34 & 99 \\
			$L^{\mathcal{D}}_{n,0.5}$ & 5 & 83 & 96 & 45 & 93 & 30 & 30 & 31 & 63 & 82 & 79 &  93 & 69 &  72 & 100 & 32 &  99\\
			$L^{\mathcal{D}}_{n,1}$ & 5 & 85 &  96 & 51 & 96 & 34 & 34 & 34 & 66 & 86 & 85 & 87 & 78 & 69 &  100 & 36 &  99\\
            $L^{\mathcal{D}}_{n,2}$ & 5 &   86 &  96 & 56 &  98 & 40 & 39 & 38 & 72 & 88 & 89 & 79 & 85 & 63 &  99 & 37 &  99 \\
            $L^{\mathcal{D}}_{n,5}$ & 5 &  86 & 95 & 59 &  99 & 42 & 43 & 44 & 74 & 90 & 92 & 65 & 89 & 56 &  99 & 39 & 98\\
            $L^{\mathcal{D}}_{n,10}$ & 5 & 84 & 95 &  62 &  99 & 45 & 43 & 45 &  76 &  91 & 92 & 62 & 80 & 53 &  99 & 39 & 98\\
            $L^{\mathcal{D}}_{n,\widehat{a}}$ & 5& 83 & 95 & 57 & 96 & 40 & 40 & 40 & 72 & 88 & 89 & 96 & 84 & 72 & 100 & 35 & 99\\
         \end{tabular}
	}
	\label{fig: comparison50}
\end{table}

We can see from tables  that all the sizes of our tests are equal to the level of significance, and that the powers range from reasonable to high. In comparison to the other exponentiality tests (see \cite{cuparic2018new} and \cite{torabi2018wide}) we can conclude that our tests are serious competitors to the most powerful classical and recent exponentiality tests.

\section{Conclusion}
In this paper we proposed two new consistent scale-free  tests for the exponential distribution. In addition, we performed an extensive comparison of efficiency of recent and classical exponentiality tests. 

We showed that our tests are very efficient and powerful and can be considered as serious competitors to  other high quality exponentilaity tests. 

From the comparison study, the general conclusion is that there is no uniformly best test, since the performance is different for different alternatives. However, the tests based on integral transforms, due to their flexibility because of the tuning parameter, generally tend to have higher efficiency, and they are recommended to use.

\section*{Appendix A -- Proofs of theorems}
\begin{proof}[Proof of Theorem \ref{raspodelaM}]
Our statistic $M_{n,a}(\widehat{\lambda}_n)$ can be  rewritten as
\begin{equation*}
\begin{aligned}
M_{n,a}(\widehat{\lambda}_n)&=\int_0^{\infty}\left(\frac{1}{n^{2}}\sum_{i_1,i_2=1}^n\xi(X_{i_1},X_{i_2},t;a\widehat{\lambda}_n)\right)^2e^{-at}dt\\
&=\int_0^{\infty}V_n(t,\widehat{\lambda}_n)^2e^{-at}dt.
\end{aligned}
\end{equation*}
Here $V_n(t;\widehat{\lambda}_n)$, for each $t>0$,  is a $V$-statistic of order 2 with an estimated parameter, and kernel  $\xi(X_{i_1},X_{i_2},t;a,\widehat{\lambda}_n)$.

Since the function $\xi(x_{1},x_{2},t;a\gamma)$ is continuously differentiable with respect to $\gamma$ at the point  $\gamma=\lambda$ we may apply the mean-value theorem. We have  
\begin{equation*}
V_n(t;\widehat{\lambda}_n)=V_n(t;\lambda)+(\widehat{\lambda}_n-\lambda)\frac{\partial V_n(t;\gamma)}{\partial\gamma}|_{\gamma=\lambda^*},
\end{equation*}
for some $\lambda^*$ between $\lambda$ and  $\widehat{\lambda}_n$.
From the Law of large numbers for V-statistics \cite[6.4.2.]{Serfling}, the partial derivative
$\frac{\partial V_n(t;\gamma)}{\partial\gamma}$ converges to
\begin{equation*}
E\left(2t\min\{X_1,X_2\}e^{-2t\min\{X_{1},X_{2}\}\gamma}-tX_{1}e^{-tX_{1}\gamma}\right)=0.
\end{equation*}
Since $\sqrt{n}(\widehat{\lambda}_n-\lambda)$ is stochastically bounded,
it follows that statistics $\sqrt{n}V_n(t;\widehat{\lambda}_n)$ and $\sqrt{n}V_n(t;1)$ are  asymptotically equally distributed.  Therefore, $nM_{n,a}(\widehat{\lambda}_n)$ and $nM_{n,a}(\lambda)$ will have the same limiting distribution. Hence we need to  derive limiting distribution of $nM_{n,a}(\lambda)$.

First notice that  $M_{n,a}(\lambda)$ is a $V$-statistic with symmetric kernel $H$. Also, since the distribution of $M_{n,a}(\lambda)$ does not depend on $\lambda$ we may assume that $\lambda=1.$

It is easy to show that its first projection of kernel $H$ on $X_1$ is equal to zero. After some calculations, we obtain that its second projection on $(X_1,X_2)$ is given by
 \begin{align*}
    \widetilde{h}_2(u,v;a)&=
    E(H(X_1,X_2,X_3,X_4;a,1)|X_1=u,X_2=v)\\&=
   \frac{1}{6}\bigg(3+\frac{1}{a+u+v}-\frac{2e^{-u}}{a+2u+v}-\frac{2e^{-v}}{a+u+2v}-(4-a)e^aEi(-a)\\&
   +e^\frac{a+v}{2}\Big(Ei\big(-\frac{a+v}{2}\big)-Ei\big(-\frac{a+2u+v}{2}\big)\Big)+e^{a+u}\Big(4Ei(-a-2u)-Ei(-a-u)\Big)\\&
   +e^\frac{a+u}{2}\Big(Ei\big(-\frac{a+u}{2}\big)-Ei\big(-\frac{a+u+2v}{2}\big)\Big)
  +e^{a+v}\Big(4Ei(-a-2v)-Ei(-a-v)\Big)\\& 
  +\frac{e^{-u-v}}{a+2(u+v)}(2a+4(1+u+v))-2(e^{-u}+e^{-v})+e^{\frac{a}{2}}\Big(-(4+a+2u)Ei(-\frac{a}{2}-u)\\&+(a+4)Ei(-\frac{a}{2})+(a+2(2+u+v))Ei(-\frac{a}{2}-u-v)
  -(4+a+2v)Ei(-\frac{a}{2}-v)\Big)\bigg),
\end{align*}
where $\text{Ei}(x)=-\int_{-x}^\infty \frac{e^{-t}}{t}dt$ is the exponential integral. The function $\widetilde{h}_2$ is  non-constant for any $a>0$. Hence,  kernel $h$ is degenerate with degree 2. 

Since the  kernel $H$ is bounded and  degenerate, from the theorem on asymptotic distribution of U-statistics with
degenerate kernels \cite[Corollary 4.4.2]{korolyuk}, and the Hoeffding representation of $V$-statistics, we get that, $M_{n,a}(1)$, being a $V$-statistic of degree 2, has the following asymptotic distribution 
\begin{equation}\label{raspodelaT}
 	nM_{n,a}(1)\overset{d}{\rightarrow}6\sum_{k=1}^\infty\delta_kW^2_k,
 	\end{equation}
where $\{\delta_k\}$  are the eigenvalues of the integral operator $\mathcal{M}_a$ defined by  
 	\begin{equation} \label{operatorA}
 	\mathcal{M}_{a}q(x)=\int_{0}^{+\infty}\widetilde{h}_2(x,y;a)q(y)dF(y),
 	\end{equation}
 	  and  $\{W_{k}\}$ is the sequence  of i.i.d.  standard Gaussian random variables.
 \end{proof}

\begin{proof}[Proof of Theorem \ref{raspodelaL}]
The test statistic $L^{\mathcal{D}}_{n,a}$ can be represented as $\sup\limits_{t\geq 0}|V_n(t;\widehat{\lambda})e^{-at}|$, where $\{V_n(t;\widehat{\lambda})\}$ is a $V-$empirical process introduced in the proof of Theorem \ref{raspodelaM}.  
We have shown that statistics $\sqrt{n}V_n(t;\widehat{\lambda}_n)$ and $\sqrt{n}V_n(t;\lambda)$ are asymptotically equally distributed, and that their distribution does not depend on $\lambda$.
Hence, $\sqrt{n}V_n(t\widehat{\lambda}_n)e^{-at}$ converges in $D(0,\infty)$ to a centered Gaussian process $\{\eta(t)\}$ (see  \cite{silverman1983convergence}), with  covariance  function
\begin{equation*}
\begin{aligned}
        K(s,t)&=e^{-a(s+t)}\int\limits_0^\infty\int\limits_0^\infty(e^{-tx}-e^{-t2\min\{x,y\}})(e^{-sx}-e^{-s2\min\{x,y\}})e^{-x-y}dxdy\\&=\frac{e^{-a(s+t)}st(4+8s+4s^2+8t+15st+6s^2t+4t^2+6st^2)}{4(1+s)(1+t)(1+s+t)(2+2s+t)(2+s+2t)(3+2s+2t)}.
\end{aligned}
\end{equation*}
Therefore $L^{\mathcal{D}}_{n,a}$ converges to $\sup_{t>0}|\eta(t)|$. This completes the proof.

\end{proof}

\begin{proof}[Proof of Lemma \ref{Slopovi}]
Using the result of \cite{Zolotarev}, the logarithmic tail behavior
of limiting distribution  function of $\widetilde{M}_{n,a}(\widehat{\lambda}_n)=\sqrt{nM_{n,a}(\widehat{\lambda}_n)}$ is

\begin{equation*}
    \log(1-F_{\widetilde{M}_a}(t))=-\frac{t^2}{12\delta_1}+o(t^2),\;\;t\to \infty.
\end{equation*}
Therefore,  $a_{\widetilde{M}_a}=\frac{1}{6\delta_1}.$
The limit in probability $P_{\theta}$ of $\widetilde{M}_{n,a}(\widehat{\lambda}_n)/\sqrt{n}$ is

\begin{equation*}
b_{\widetilde{M}_{a}}=\sqrt{b_{M}(\theta)}.    
\end{equation*}

The expression for $b_M(\theta)$ is derived in the following lemma.

\begin{lemma}\label{lema_B}
	For a given alternative density $g(x;\theta)$ whose distribution belongs to $\mathcal{G}$, we have that the limit in probability of the statistic $M_{n,a}(\widehat{\lambda}_n)$ is 	
	
	\begin{equation*}
	b_M(\theta)=6\int\limits_{0}^{\infty}\int\limits_{0}^{\infty}\widetilde{h}_2(x,y;a)g'_{\theta}(x;0)g'_{\theta}(y;0)dxdy\cdot\theta^2+o(\theta^2), \theta\rightarrow0.
	\end{equation*}

\end{lemma}

\begin{proof}
For brevity,  denote  $\boldsymbol{x}=(x_1,x_2,x_3,x_{4})$ and $\boldsymbol{G}(\boldsymbol{x};\theta)=\prod_{i=1}^{4}G(x_i;\theta)$.
Since $\overline{X}_n$ converges almost surely to its
expected value $\mu(\theta)$, using the Law of large numbers for $V$-statistics with estimated parameters (see \cite{iverson}),  $M^{\mathcal{D}}_{n,a}(\widehat{\lambda}_n)$ converges to
\begin{equation*}
\begin{aligned}
	b_M(\theta)&=E_{\theta}(H(\boldsymbol{X},a;\mu(\theta)))\\&=\int_{{(R^+)}^{4}}\Big(\frac{\mu(\theta)}{x_1+x_{3}+a\mu(\theta)}-\frac{\mu(\theta)}{x_{3}+2\min\{x_{1},x_{2}\}+a\mu(\theta)}\\&-\frac{\mu(\theta)}{x_{1}+2\min\{x_{3},x_{4}\}+a\mu(\theta)}+\frac{\mu(\theta)}{2\min\{x_{1},x_{2}\}+2\min\{x_{3},x_{4}\}+a\mu(\theta)}\Big)d\boldsymbol{G}(\boldsymbol{x};\theta).
\end{aligned}
\end{equation*}
We may assume that $\mu(0)=1$ since the test statistic is ancillary for $\lambda$ under the null hypothesis. After some calculations we get that $b'_M(0)=0$ and that
\begin{equation*}
\begin{split}
b''(0)&
=\int_{(R^+)^{4}}H(\boldsymbol{x},a;1)\frac{\partial^2}{\partial\theta^2}d\boldsymbol{G}(\boldsymbol{x},0)
=6\int_{(R^+)^2}\widetilde{h}_2(x,y)g'_{\theta}(x;0)g'_{\theta}(y;0)dxdy.
\end{split}
\end{equation*}
Expanding $b_M(\theta)$ into the Maclaurin series we complete the proof. 
\end{proof}


Now we pass to the statistic $L^{\mathcal{D}}_n.$
The tail behaviour of the random variable $\sup_{t>0}|n_{t}|$ is equal to the inverse of supremum of its covariance function, i.e. the $a_L=\frac{1}{\sup_{t>0}K(t,t)}$ (see \cite{marcus1972sample}).


Similarly like before, since $\overline{X}_n$ converges almost surely to its expected value $\mu(\theta)$, using the Law of large numbers for $V$-statistics with estimated parameters (see \cite{iverson}),  $V_n(t,a;\hat{\lambda})e^{-at}$ converges to
\begin{equation*}
\begin{aligned}
	b_L(\theta)&=E_{\theta}(\Phi(X_1,X_2;t,a,\mu(\theta))).
\end{aligned}
\end{equation*}
Expanding $b_L(\theta)$ in the Maclaurin series we obtain
\begin{equation*}
    b_L(\theta)=2\int_{0}^{\infty}\widetilde{\varphi}_1(x,t;a)g'_{\theta}(x;0)dx\cdot\theta+o(\theta),
\end{equation*}
where $\widetilde{\varphi}_1(x,t;a)=E(\Phi(X_1,X_2,t;a,1)|X_1=x_1).$
According to the Glivenko-Cantelli theorem for V-statistics (\cite{helmers1988glivenko}) the limit in probability under the alternative for statistics $L_{n,a}^{\mathcal{D}}$ is equal to $\sup_{t\geq0}|b_L(\theta)|$. Inserting this into the expression for the Bahadur slope completes the proof. 
\end{proof}

\section*{Appendix B -- Bahadur approximate slopes}

\begin{proof}{Approximate local Bahadur slope of statistics EP and CO }

Those statistics can be represented as
$$T_n=\frac{1}{n}\sum_{i=1}^n\Phi(X;\hat{\mu}),$$ where $\Phi(x;\gamma)$ is continuously   differentiable with respect to $\gamma$ at point $\gamma=\mu.$ 
It was shown that the limiting distribution of $\sqrt{n}T_n$ is zero mean normal  with variance $\sigma^2_{\Phi}$ (see \cite{epps1986test} and \cite{cox1984analysis}). Hence, the coefficient $a_T$ is equal to $\frac{1}{\sigma^2_{\Phi}}.$

Further, we have
\begin{equation*}
\begin{aligned}
	b(\theta)&=E_{\theta}(\Phi(X;\mu(\theta)))=\int\limits_0^\infty \Phi(x;\mu(\theta))dG(x;\theta)
\end{aligned}
\end{equation*}
\begin{equation*}
\begin{aligned}
	b'(\theta)&=\int\limits_0^\infty \frac{\partial}{\partial\mu}\Phi(x;\mu(\theta))\frac{\partial}{\partial\theta}\mu(\theta)dG(x;\theta)+\int\limits_0^\infty \Phi(x;\mu(\theta))\frac{\partial}{\partial\theta}dG(x;\theta).\\
\end{aligned}
\end{equation*}
Then it holds that
\begin{equation*}
\begin{aligned}
	b(\theta)&=b(0)+b'(0)\theta+o(\theta)\\&=\Bigg(\mu'(0)\int\limits_0^\infty \Phi'(x;1)g(x;0)dx+\int\limits_0^\infty\Phi(x;1)g'(x;0)dx)\bigg)\theta+o(\theta).
\end{aligned}
\end{equation*}
From this we obtain the expression for $c_T(\theta).$
\end{proof}



\begin{proof}{Approximate local Bahadur slope of statistics BH, HE, Wn, HM, $\omega^2$  and AD}

Let $T$ be the one of considered statistics.
It was shown that  the limiting distribution of $nT_n$ is $\sum_{i=1}^{\infty}\delta_iW_i^2,$ where $\{W_i\}$ is the sequence of i.i.d. standard normal variables and $\{\delta_i\}$ the sequence of eigenvalues of certain covariance operator.  
Using the result of Zolotarev in \cite{Zolotarev}, we have that the logarithmic tail behavior of limiting distribution function of $\widetilde{T}_n=\sqrt{n T_n}$ is 
\begin{equation*}
    \log(1-F_{\widetilde{T}}(s))=-\frac{s^2}{2\delta_1}+o(s^2), s\rightarrow\infty.
\end{equation*}
Next, the limit in probability of $\widetilde{T}_n/\sqrt{n}$ is $b_{\widetilde{T}}(\theta)=\sqrt{b_T(\theta)}$.
Statistic $T_n$ can be represented as
\begin{equation*}
T_{n}=\frac{1}{n^2}\sum_{k,j=1}^n\Phi(X_k,X_j;\hat{\mu}).
\end{equation*}

As before,  we may assume that $\mu(0)=1$.
Since the sample mean converges almost surely to its expected value, by using the Law of large numbers for
$V$-statistics with estimated parameters (see \cite{iverson}), we can conclude that the limit in the probability of statistic $T_n$ is equal to the one of 
\begin{equation*}
b_T(\theta)=E_{\theta}(\Phi(X_1,X_2;\mu(\theta)))=\int\limits_{0}^\infty\int\limits_{0}^\infty \Phi(x,y;\mu(\theta))g(x;\theta)g(y;\theta)dxdy.
\end{equation*}
We get that  $b_T'(0)=0$ and that 
\begin{equation*}
\begin{split}
b_T''(0)=2\int\limits_{0}^\infty\int\limits_{0}^\infty& \Phi(x,y;1)g'_{\theta}(x;0)g'_{\theta}(y;0)dxdy+4\mu'(0)\int\limits_{0}^\infty\int\limits_{0}^\infty \Phi'(x,y;1)g(x;0)g'_{\theta}(y;0)dxdy\\&+(\mu'(0))^2\int\limits_{0}^\infty\int\limits_{0}^\infty\Phi''(x,y;1)g(x;0)g(y;0)dxdy,
\end{split}
\end{equation*}
Expanding $b_T(\theta)$ into Maclaurin series 
we obtain expression for $b_T$.
\end{proof}

\section*{Appendix C -- Tables of efficiencies}

\begin{table}[!htbp]
	\centering
	\caption{Relative Bahadur efficiency with respect to LRT}
\label{fig: efikasnosti1}
\begin{tabular}{ccccc}
   & $Weibull$ & $Gamma$ & $LFR$ & $EMNW(3)$ \\\hline
    $EP_n$ & 0.876 & 0.694 & 0.750 & 0.937\\
    $CO_n$ & 1 & 0.943 & 0.326 & 0.917 \\
    $G_n$ & 0.876 & 0.694 & 0.750 & 0.937\\
    $MO_n$ & 0.943 & 1 & 0.388 & 0.814 \\
    $I_{n,2}^{(1)}$ & 0.621 & 0.723 & 0.104 & 0.694\\
    $I_{n,3}^{(1)}$ & 0.664 & 0.708 & 0.159 & 0.799\\
    $I_{n}^{(2)}$ & 0.750 & 0.796 & 0.208 & 0.844\\
    $I_{n}^{(3)}$ & 0.746 & 0.701 & 0.308 & 0.916\\
    $I_{n}^{(4)}$ & 0.649 & 0.638 & 0.206 & 0.835\\
    $I_{n}^{\mathcal{P}}$ & 0.821 & 0.788 & 0.337 & 0.949\\
    $I_{n}^{\mathcal{D}}$ & 0.697 & 0.790 & 0.149 & 0.746\\
    
    $J_{n,0.2}^{\mathcal{P}}$ &0.750 & 0.856 & 0.171 &0.751\\
    $J_{n,0.5}^{\mathcal{P}}$ &0.812 & 0.843 & 0.262 & 0.888\\
    $J_{n,1}^{\mathcal{P}}$ &0.846 & 0.820 & 0.349 & 0.955\\
    $J_{n,2}^{\mathcal{P}}$ & 0.868 & 0.792 & 0.445 & 0.985\\
    $J_{n,5}^{\mathcal{P}}$ &0.882 & 0.756 & 0.566 & 0.987 \\
    $J_{n,10}^{\mathcal{P}}$ & 0.884 & 0.733 & 0.637 & 0.974 \\
    
    $J_{n,0.2}^{\mathcal{D}}$ &0.526 & 0.731 & 0.053 & 0.370\\
    $J_{n,0.5}^{\mathcal{D}}$ & 0.674 & 0.826 & 0.117 & 0.608\\
    $J_{n,1}^{\mathcal{D}}$ & 0.771 & 0.857 & 0.198 & 0.786\\
    $J_{n,2}^{\mathcal{D}}$ & 0.842 & 0.854 & 0.305 & 0.917\\
    $J_{n,5}^{\mathcal{D}}$ &0.889 & 0.813 & 0.465 & 0.991\\
    $J_{n,10}^{\mathcal{D}}$ & 0.896 & 0.775 & 0.569 & 0.994\\
    
    $KS$ & 0.538 & 0.503 & 0.356 & 0.686\\
    
    $D_{n,2}^{(1)}$ & 0.092 & 0.093 & 0.052 & 0.149 \\
    $D_{n,3}^{(1)}$ & 0.152 & 0.138 & 0.106& 0.230\\
    $D_{n}^{(2)}$ & 0.277 & 0.267 & 0.155 & 0.396\\
    $D_{n}^{(3)}$ & 0.258 & 0.212 & 0.213 & 0.364\\
    $D_{n}^{(4)}$ & 0.079 & 0.066 & 0.067 & 0.122\\
    $D_{n}^{\mathcal{P}}$ & 0.437 & 0.448 & 0.192 & 0.592\\
    $D_{n}^{\mathcal{D}}$ & 0.158 & 0.174 & 0.073 & 0.247\\
    
    $\omega^2_n$ & 0.808 & 0.701 & 0.588 & 0.958\\
    $AD_n$ & 0.909 & 0.863 & 0.573 & 0.996 \\
    
    $BH_{n,0.2}$ & 0.905 & 0.928 & 0.421 & 0.914\\
    $BH_{n,0.5}$ & 0.932& 0.877 & 0.534 & 0.987\\
    $BH_{n,1}$ & 0.926& 0.810 & 0.638 & 0.996\\
    $BH_{n,2}$ & 0.894& 0.726 & 0.749 & 0.956\\
    $BH_{n,5}$ & 0.823& 0.611 & 0.878 & 0.848\\
    $BH_{n,10}$ & 0.771& 0.542 & 0.956 & 0.767\\
  $HE_{n,0.2}$ & 0.923 & 0.928 & 0.420 & 0.927\\
    $HE_{n,0.5}$ & 0.940 & 0.868 & 0.542 & 0.991\\
    $HE_{n,1}$ & 0.928 & 0.799 & 0.647 & 0.992\\
     $HE_{n,2}$ & 0.893 & 0.719 & 0.752 & 0.949\\
\end{tabular}
\end{table}

	\begin{table}[!htbp]
	\centering
	\caption{Relative Bahadur efficiency with respect to LRT}
\label{fig: efikasnosti2}
\begin{tabular}{ccccc}
   & $Weibull$ & $Gamma$ & $LFR$ & $EMNW(3)$ \\\hline

    $HE_{n,5}$ & 0.822 & 0.609 & 0.873 & 0.846\\
    $HE_{n,10}$ & 0.761 & 0.536 & 0.935 &0.758\\
    
    $W_{n,0.2}$ & 0.790 & 0.914 & 0.224 & 0.688\\
    $W_{n,0.5}$ & 0.905 & 0.922 & 0.382 & 0.909\\
    $W_{n,1}$ & 0.935 & 0.864 & 0.528 & 0.991\\
    $W_{n,2}$ & 0.917 & 0.772 & 0.677 & 0.983\\
    $W_{n,5}$ & 0.842 & 0.638 & 0.842 & 0.877\\
    $W_{n,10}$ & 0.774 & 0.550 & 0.924 & 0.776\\
    
    $HM_{n,0.2}^{(1)}$ & 0.324 & 0.448 & 0.049 & 0.271\\
    $HM_{n,0.5}^{(1)}$ & 0.560 & 0.621 & 0.174 & 0.643\\
    $HM_{n,1}^{(1)}$ & 0.691 & 0.642 & 0.361 & 0.865\\
    $HM_{n,2}^{(1)}$ & 0.715 & 0.557 & 0.612 & 0.855\\
    $HM_{n,5}^{(1)}$ & 0.591 & 0.373 & 0.895 & 0.582\\
    $HM_{n,10}^{(1)}$ & 0.452 & 0.254 & 0.951 & 0.382\\
    
    $HM_{n,0.2}^{(2)}$ &0.633 & 0.579 & 0.320 & 0.818\\
    $HM_{n,0.5}^{(2)}$ & 0.673& 0.533 & 0.520 & 0.828\\
    $HM_{n,1}^{(2)}$ & 0.656& 0.468 & 0.683 & 0.742\\
    $HM_{n,2}^{(2)}$ & 0.603& 0.391 & 0.825 & 0.616\\
    $HM_{n,5}^{(2)}$ & 0.504& 0.295& 0.931 & 0.451\\
    $HM_{n,10}^{(2)}$ & 0.430& 0.238 & 0.942 & 0.355\\
    
   $M_{n,0.2}^{(\mathcal{P})}$ & 0.734 & 0.832 & 0.183 & 0.754\\
   $M_{n,0.5}^{(\mathcal{P})}$ & 0.787 & 0.827 & 0.253 & 0.865\\
   $M_{n,1}^{(\mathcal{P})}$ & 0.822 & 0.814 & 0.324 & 0.929\\
   $M_{n,2}^{(\mathcal{P})}$ & 0.850 & 0.794 & 0.407 & 0.969\\
   $M_{n,5}^{(\mathcal{P})}$ & 0.873 & 0.764 & 0.523 &  0.985\\
   $M_{n,10}^{(\mathcal{P})}$ & 0.881 & 0.742 & 0.601 & 0.979\\
   $M_{n,0.2}^{(\mathcal{D})}$ & 0.533 & 0.712 & 0.080 & 0.443\\
   $M_{n,0.5}^{(\mathcal{D})}$ & 0.645 & 0.788 & 0.130 & 0.610\\
   $M_{n,1}^{(\mathcal{D})}$ & 0.729 & 0.825 & 0.191 & 0.750\\
   $M_{n,2}^{(\mathcal{D})}$ & 0.803 & 0.838 & 0.275 & 0.873\\
   $M_{n,5}^{(\mathcal{D})}$ & 0.867 & 0.820 & 0.413 & 0.971\\
   $M_{n,10}^{(\mathcal{D})}$ & 0.889 & 0.789 & 0.520 & 0.992\\
   $L_{n,0.2}^{(\mathcal{D})}$ & 0.738 & 0.798 & 0.259 & 0.821\\
   $L_{n,0.5}^{(\mathcal{D})}$ & 0.799 & 0.815 & 0.323 & 0.902\\
   $L_{n,1}^{(\mathcal{D})}$ & 0.844 & 0.815 & 0.394 & 0.957\\
   $L_{n,2}^{(\mathcal{D})}$ & 0.875 & 0.800 & 0.479 & 0.988\\
   $L_{n,5}^{(\mathcal{D})}$ & 0.892 & 0.766 & 0.588 & 0.990 \\
   $L_{n,10}^{(\mathcal{D})}$ & 0.891 & 0.740 & 0.652 &0.976 \\
\end{tabular}
\end{table}



\section*{Acknowledgement}
This work was supported by the MNTRS, Serbia under  Grant No. 174012 (first  and second  author).

\bibliographystyle{abbrv}
\bibliography{literatura}

\begin{thebibliography}{10}

\bibitem{alizadeh2011testing}
H.~Alizadeh~Noughabi and N.~R. Arghami.
\newblock Testing exponentiality based on characterizations of the exponential
  distribution.
\newblock {\em Journal of Statistical Computation and Simulation},
  81(11):1641--1651, 2011.

\bibitem{allison}
J.~Allison and L.~Santana.
\newblock On a data-dependent choice of the tuning parameter appearing in
  certain goodness-of-fit tests.
\newblock {\em Journal of Statistical Computation and Simulation},
  85(16):3276--3288, 2015.

\bibitem{allison2017apples}
J.~Allison, L.~Santana, N.~Smit, and I.~Visagie.
\newblock An 'apples to apples' comparison of various tests for exponentiality.
\newblock {\em Computational Statistics}, 32(4):1241--1283, 2017.

\bibitem{arnold2013exponential}
B.~C. Arnold and J.~A. Villasenor.
\newblock Exponential characterizations motivated by the structure of order
  statistics in samples of size two.
\newblock {\em Statistics \& Probability Letters}, 83(2):596--601, 2013.

\bibitem{bahadur1960asymptotic}
R.~R. Bahadur.
\newblock On the asymptotic efficiency of tests and estimates.
\newblock {\em Sankhy{\=a}: The Indian Journal of Statistics}, pages 229--252,
  1960.

\bibitem{bahadur1967rates}
R.~R. Bahadur.
\newblock Rates of convergence of estimates and test statistics.
\newblock {\em The Annals of Mathematical Statistics}, 38(2):303--324, 1967.

\bibitem{baringhaus1991class}
L.~Baringhaus and N.~Henze.
\newblock A class of consistent tests for exponentiality based on the empirical
  {L}aplace transform.
\newblock {\em Annals of the Institute of Statistical Mathematics},
  43(3):551--564, 1991.

\bibitem{bozin}
V.~Bo\v{z}in, B.~Milo\v{s}evi\'c, {\relax Ya}.~{\relax Yu}. Nikitin, and
  M.~Obradovi\'c.
\newblock New characterization based symmetry tests.
\newblock {\em Bulletin of the Malaysian Mathematical Sciences Society}, 2018.
\newblock DOI:10.1007/s40840-018-0680-3.

\bibitem{cox1984analysis}
D.~Cox and D.~Oakes.
\newblock {\em Analysis of survival data}.
\newblock Chapman and Hall, New York, 1984.

\bibitem{cuparic2018new}
M.~Cupari{\'c}, B.~Milo{\v{s}}evi{\'c}, and M.~Obradovi{\'c}.
\newblock New $l^{2}$-type exponentiality tests.
\newblock {\em SORT}, 2018.
\newblock accepted for publication.

\bibitem{desu1971}
M.~M. Desu.
\newblock A characterization of the exponential distribution by order
  statistics.
\newblock {\em The Annals of Mathematical Statistics.}, 42(2):837--838, 1971.

\bibitem{epps1986test}
T.~Epps and L.~Pulley.
\newblock A test of exponentiality vs. monotone-hazard alternatives derived
  from the empirical characteristic function.
\newblock {\em Journal of the Royal Statistical Society. Series B
  (Methodological)}, pages 206--213, 1986.

\bibitem{gail1978scale}
M.~Gail and J.~Gastwirth.
\newblock A scale-free goodness-of-fit test for the exponential distribution
  based on the {G}ini statistic.
\newblock {\em Journal of the Royal Statistical Society: Series B
  (Methodological)}, 40(3):350--357, 1978.

\bibitem{grane2009location}
A.~Gran{\'e} and J.~Fortiana.
\newblock A location-and scale-free goodness-of-fit statistic for the
  exponential distribution based on maximum correlations.
\newblock {\em Statistics}, 43(1):1--12, 2009.

\bibitem{grane2011directional}
A.~Gran{\'e} and J.~Fortiana.
\newblock A directional test of exponentiality based on maximum correlations.
\newblock {\em Metrika}, 73(2):255--274, 2011.

\bibitem{helmers1988glivenko}
R.~Helmers, P.~Janssen, and R.~Serfling.
\newblock Glivenko-{C}antelli properties of some generalized empirical df's and
  strong convergence of generalized {L}-statistics.
\newblock {\em Probability theory and related fields}, 79(1):75--93, 1988.

\bibitem{henze1992new}
N.~Henze.
\newblock A new flexible class of omnibus tests for exponentiality.
\newblock {\em Communications in Statistics-Theory and Methods},
  22(1):115--133, 1992.

\bibitem{henze2002tests}
N.~Henze and S.~Meintanis.
\newblock Tests of fit for exponentiality based on the empirical {L}aplace
  transform.
\newblock {\em Statistics: A Journal of Theoretical and Applied Statistics},
  36(2):147--161, 2002.

\bibitem{henze2002goodness}
N.~Henze and S.~G. Meintanis.
\newblock Goodness-of-fit tests based on a new characterization of the
  exponential distribution.
\newblock {\em Communications in Statistics-Theory and Methods},
  31(9):1479--1497, 2002.

\bibitem{henze2005}
N.~Henze and S.~G. Meintanis.
\newblock Recent and classical tests for exponentiality: a partial review with
  comparisons.
\newblock {\em Metrika}, 61(1):29--45, 2005.

\bibitem{iverson}
H.~Iverson and R.~Randles.
\newblock The effects on convergence of substituting parameter estimates into
  {U}-statistics and other families of statistics.
\newblock {\em Probability Theory and Related Fields}, 81(3):453--471, 1989.

\bibitem{jovanovic}
M.~Jovanovi{\'c}, B.~Milo{\v{s}}evi{\'c}, {\relax Ya}.~{\relax Yu}. Nikitin,
  M.~Obradovi{\'c}, and {\relax K}.~{\relax Yu.}. Volkova.
\newblock Tests of exponentiality based on {A}rnold--{V}illasenor
  characterization and their efficiencies.
\newblock {\em Computational Statistics \& Data Analysis}, 90:100--113, 2015.

\bibitem{klar2003test}
B.~Klar.
\newblock On a test for exponentiality against {L}aplace order dominance.
\newblock {\em Statistics}, 37(6):505--515, 2003.

\bibitem{klar2005tests}
B.~Klar.
\newblock Tests for exponentiality against the {M} and {LM-C}lasses of life
  distributions.
\newblock {\em Test}, 14(2):543--565, 2005.

\bibitem{korolyuk}
V.~S. Korolyuk and Y.~V. Borovskikh.
\newblock {\em Theory of {U}-statistics}.
\newblock Kluwer, Dordrecht, 1994.

\bibitem{marcus1972sample}
M.~B. Marcus and L.~Shepp.
\newblock Sample behavior of {G}aussian processes.
\newblock In {\em Proc. of the Sixth Berkeley Symposium on Math. Statist. and
  Prob}, volume~2, pages 423--421, 1972.

\bibitem{meintanis2008tests}
S.~G. Meintanis.
\newblock Tests for generalized exponential laws based on the empirical
  {M}ellin transform.
\newblock {\em Journal of Statistical Computation and Simulation},
  78(11):1077--1085, 2008.

\bibitem{meintanis2007testing}
S.~G. Meintanis, {\relax Ya}.~{\relax Yu}. Nikitin, and A.~Tchirina.
\newblock Testing exponentiality against a class of alternatives which includes
  the {RNBUE} distributions based on the empirical {L}aplace transform.
\newblock {\em Journal of Mathematical Sciences}, 145(2):4871--4879, 2007.

\bibitem{bojanaMetrika}
B.~Milo{\v{s}}evi{\'c}.
\newblock Asymptotic efficiency of new exponentiality tests based on a
  characterization.
\newblock {\em Metrika}, 79(2):221--236, 2016.

\bibitem{MilosevicObradovicPapers}
B.~Milo{\v{s}}evi{\'c} and M.~Obradovi{\'c}.
\newblock New class of exponentiality tests based on {U}-empirical {L}aplace
  transform.
\newblock {\em Statistical Papers}, 57(4):977--990, 2016.

\bibitem{Publ}
B.~Milo{\v{s}}evi{\'c} and M.~Obradovi{\'c}.
\newblock Some characterization based exponentiality tests and their {B}ahadur
  efficiencies.
\newblock {\em Publications de L'Institut Mathematique}, 100(114):107--117,
  2016.

\bibitem{milovsevic2016some}
B.~Milo{\v{s}}evi\'c and M.~Obradovi\'c.
\newblock Some characterizations of the exponential distribution based on order
  statistics.
\newblock {\em Applicable Analysis and Discrete Mathematics}, 10(2):394--407,
  2016.

\bibitem{moran1951random}
P.~Moran.
\newblock The random division of an interval -- {P}art {II}.
\newblock {\em Journal of the Royal Statistical Society: Series B
  (Methodological)}, 13(1):147--150, 1951.

\bibitem{nikitin2010large}
Y.~Y. Nikitin.
\newblock Large deviations of {U}-empirical {K}olmogorov--{S}mirnov tests and
  their efficiency.
\newblock {\em Journal of Nonparametric Statistics}, 22(5):649--668, 2010.

\bibitem{nikitin1996bahadur}
Y.~Y. Nikitin and A.~Tchirina.
\newblock Bahadur efficiency and local optimality of a test for the exponential
  distribution based on the {G}ini statistic.
\newblock {\em Journal of the Italian Statistical Society}, 5(1):163--175,
  1996.

\bibitem{nikitin2007lilliefors}
Y.~Y. Nikitin and A.~Tchirina.
\newblock Lilliefors test for exponentiality: large deviations, asymptotic
  efficiency, and conditions of local optimality.
\newblock {\em Mathematical Methods of Statistics}, 16(1):16--24, 2007.

\bibitem{nikitinKnjiga}
{\relax Ya}.~{\relax Yu}. Nikitin.
\newblock {\em Asymptotic efficiency of nonparametric tests}.
\newblock Cambridge University Press, New York, 1995.

\bibitem{nikitinMetron}
{\relax Ya}.~{\relax Yu}. Nikitin and I.~Peaucelle.
\newblock Efficiency and local optimality of nonparametric tests based on {U}-
  and {V}-statistics.
\newblock {\em Metron}, 62(2):185--200, 2004.

\bibitem{NikVol}
{\relax Ya}.~{\relax Yu}. Nikitin and K.~{\relax Yu.}. Volkova.
\newblock Asymptotic efficiency of exponentiality tests based on order
  statistics characterization.
\newblock {\em Georgian Mathematical Journal}, 17(4):749--763, 2010.

\bibitem{nikitin2016efficiency}
{\relax Ya}.~{\relax Yu}. Nikitin and {\relax K}.~{\relax Yu}. Volkova.
\newblock Efficiency of exponentiality tests based on a special property of
  exponential distribution.
\newblock {\em Mathematical Methods of Statistics}, 25(1):54--66, 2016.

\bibitem{obradovic2014three}
M.~Obradovi{\'c}.
\newblock Three characterizations of exponential distribution involving median
  of sample of size three.
\newblock {\em Journal of Statistical Theory and Applications}, 14(3):257--264,
  2015.

\bibitem{puri1970characterization}
P.~S. Puri and H.~Rubin.
\newblock A characterization based on the absolute difference of two iid random
  variables.
\newblock {\em The Annals of Mathematical Statistics}, 41(6):2113--2122, 1970.

\bibitem{Serfling}
R.~Serfling.
\newblock {\em Approximation theorems of mathematical statistics}, volume 162.
\newblock John Wiley \& Sons, New York, 2009.

\bibitem{silverman1983convergence}
B.~Silverman et~al.
\newblock Convergence of a class of empirical distribution functions of
  dependent random variables.
\newblock {\em The Annals of Probability}, 11(3):745--751, 1983.

\bibitem{strzalkowska2017goodness}
E.~Strzalkowska-Kominiak and A.~Gran{\'e}.
\newblock Goodness-of-fit test for randomly censored data based on maximum
  correlation.
\newblock {\em SORT: statistics and operations research transactions},
  41(1):119--138, 2017.

\bibitem{tchirina2005bahadur}
A.~Tchirina.
\newblock Bahadur efficiency and local optimality of a test for exponentiality
  based on the {M}oran statistics.
\newblock {\em Journal of Mathematical Sciences}, 127(1):1812--1819, 2005.

\bibitem{tenreiro2019automatic}
C.~Tenreiro.
\newblock On the automatic selection of the tuning parameter appearing in
  certain families of goodness-of-fit tests.
\newblock {\em Journal of Statistical Computation and Simulation},
  89(10):1780--1797, 2019.

\bibitem{torabi2018wide}
H.~Torabi, N.~H. Montazeri, and A.~Gran{\'e}.
\newblock A wide review on exponentiality tests and two competitive proposals
  with application on reliability.
\newblock {\em Journal of Statistical Computation and Simulation},
  88(1):108--139, 2018.

\bibitem{volkova2015goodness}
K.~{\relax Yu}. Volkova.
\newblock Goodness-of-fit tests for exponentiality based on
  {Y}anev-{C}hakraborty characterization and their efficiencies.
\newblock {\em Proceedings of the 19th European Young Statisticians Meeting,
  Prague}, pages 156--159, 2015.

\bibitem{yanev2013characterizations}
G.~P. Yanev and S.~Chakraborty.
\newblock Characterizations of exponential distribution based on sample of size
  three.
\newblock {\em Pliska Studia Mathematica Bulgarica}, 22(1):237p--244p, 2013.

\bibitem{Zolotarev}
V.~M. Zolotarev.
\newblock Concerning a certain probability problem.
\newblock {\em Theory of Probability \& Its Applications}, 6(2):201--204, 1961.

\end{thebibliography}

\end{document}